\documentclass[10pt,a4paper]{amsart}

\usepackage[utf8]{inputenc}
\usepackage[british]{babel}
\usepackage{amsmath}
\usepackage{amsfonts}
\usepackage{amssymb}
\usepackage{amsthm}
\usepackage[pdftex]{hyperref}
\usepackage{color}

\DeclareFontFamily{U}{matha}{\hyphenchar\font45}
\DeclareFontShape{U}{matha}{m}{n}{
      <5> <6> <7> <8> <9> <10> gen * matha
      <10.95> matha10 <12> <14.4> <17.28> <20.74> <24.88> matha12
      }{}
\DeclareSymbolFont{matha}{U}{matha}{m}{n}
\DeclareFontSubstitution{U}{matha}{m}{n}

\DeclareFontFamily{U}{mathx}{\hyphenchar\font45}
\DeclareFontShape{U}{mathx}{m}{n}{
      <5> <6> <7> <8> <9> <10>
      <10.95> <12> <14.4> <17.28> <20.74> <24.88>
      mathx10
      }{}
\DeclareSymbolFont{mathx}{U}{mathx}{m}{n}
\DeclareFontSubstitution{U}{mathx}{m}{n}

\DeclareMathDelimiter{\vvvert}{0}{matha}{"7E}{mathx}{"17}

\newtheorem{main-theorem}{Theorem}
\newtheorem{proposition}{Proposition}[section]
\newtheorem{corollary}[proposition]{Corollary}
\newtheorem{lemma}[proposition]{Lemma}
\newtheorem{theorem}[proposition]{Theorem}
\newtheorem{definition}[proposition]{Definition}
\theoremstyle{remark}
\newtheorem{remark}[proposition]{Remark}
\newtheorem*{acknowledgements}{Acknowledgements}

\newcommand{\R}{\mathbb{R}}
\newcommand{\C}{\mathbb{C}}
\newcommand{\Z}{\mathbb{Z}}
\newcommand{\N}{\mathbb{N}}

\newcommand{\dd}{\mathrm{d}}
\DeclareMathOperator{\supp}{supp}
\DeclareMathOperator{\sign}{sign}

\newcommand{\marnote}[1]{\marginpar{\footnotesize #1}}

\author{Pedro Caro}
\author{Andoni Garcia}
\address{BCAM - Basque Center for Applied mathematics}
\email{pcaro@bcamath.org}
\email{agarcia@bcamath.org}

\date{\today}

\title[Scattering with critically-singular and $\delta$-shell potentials]{Scattering with critically-singular and \\ $\delta$-shell potentials}

\begin{document}

\begin{abstract}
The authors consider a scattering problem for electric potentials 
that have a component which is critically singular in the sense of 
Lebesgue spaces, and a component given by a measure supported on a compact Lipschitz hypersurface. They study direct and inverse point-source scattering
under the assumptions that the potentials are real-valued and 
compactly supported. To solve the direct scattering problem, the 
authors introduce two functional spaces ---sort of Bourgain type 
spaces--- that allow to refine the classical resolvent estimates of 
Agmon and H\"ormander, and Kenig, Ruiz and Sogge. These spaces seem to be very useful 
to deal with the critically-singular and $\delta$-shell components of 
the potentials at the same time. Furthermore, these spaces and their 
corresponding resolvent estimates turn out to have a strong 
connection with the estimates for the conjugated Laplacian used in 
the context of the inverse Calder\'on problem.
In fact, the authors derive 
the classical estimates by Sylvester and Uhlmann, and the more recent 
ones by Haberman and Tataru after some embedding properties
of these new spaces. Regarding the inverse scattering problem,
the authors prove uniqueness for the potentials from point-source 
scattering data at fix energy. To address the question of uniqueness 
the authors combine some of the most advanced techniques in the 
construction of complex geometrical optics solutions.
\end{abstract}

\maketitle


\section{Introduction}
In this paper we study a point-source scattering problem for 
electric potentials that are a combination of
\emph{critically-singular potentials} and
\emph{$\delta$-shell potentials}. More precisely, we are interested 
in real potentials of the form
\begin{equation}
V = V^0 + \alpha \, \dd \sigma
\label{id:V}
\end{equation}
where $V^0$ stands for the critically-singular component of the potential and $\alpha\, \dd\sigma$ is its $\delta$-shell component.
Here $V^0 \in L^{d/2}(\R^d; \R)$, $\sigma$ denotes the surface measure of
$\Gamma$, $\alpha \in L^\infty (\Gamma; \R)$ and $\Gamma$ is a compact hypersurface which is locally described by the graphs of Lipschitz functions.
Additionally, we assume the support of $V$ to be contained in the 
ball $B_0 = \{ x \in \R^d : |x| < R_0 \}$ with $R_0 \geq 1$.
For this class of 
potentials, we study direct and inverse point-source scattering 
in dimension $d \geq 3$. However, we carry out part of our analysis in dimension $d \geq 2$, emphasizing when $d \geq 3$ is required.

\subsection{Direct scattering} The direct scattering theory for 
potentials as $V$ follows the general 
scheme of more regular potentials. First, we consider an incident 
wave $u_{\rm in}$, which solves the equation $(\Delta + \lambda) u_{\rm in} = 0$ in $\R^d \setminus \{y\} $ with $|y| \geq R_0$. Then, 
the scattering solution $u_{\rm sc}$ solves
\[(\Delta + \lambda - V)u_{\rm sc} = V u_{\rm in} \, \, \textnormal{in}\, \R^d,\]
and satisfies the ingoing or outgoing Sommerfeld radiation condition (SRC for short). There are at least two possible ways of showing the 
existence of the scattering solution $u_{\rm sc}$. One based on a Neumann series argument, which consists of solving the problem
\begin{equation}
\left\{
\begin{aligned}
& (\Delta + \lambda)u_n = V u_{n - 1} \qquad \textnormal{in}\, \R^d,\\
& u_n \, \textnormal{satisfying SRC}
\end{aligned}
\right.
\label{pb:Neumann_iteration}
\end{equation}
for each $n \in \N$ with $u_0 = u_{\rm in}$, and showing that
$\sum_{n \in \N} u_n$ makes sense. In this case, the scattering 
solution is given by $u_{\rm sc} = \sum_{n \in \N} u_n $. The problem \eqref{pb:Neumann_iteration} can be solved using an 
appropriate 
inverse, denoted throughout the paper by
$(\Delta + \lambda \pm i 0)^{-1}$ ---the sign $\pm$ accounts for the ingoing and outgoing radiation conditions. Thus,
\[u_n = (\Delta + \lambda \pm i 0)^{-1}(V u_{n-1}) = [(\Delta + \lambda \pm i 0)^{-1} \circ V]^n (u_0),\]
and consequently, in order for $\sum_{n \in \N} u_n$ to converge, 
we only have to see that the linear operator
$(\Delta + \lambda \pm i 0)^{-1} \circ V$ is bounded in certain Banach 
space and its norm is strictly less than $1$ ---in short, it is a contraction. Here and throughout the 
article, $V$ denotes not only the potential but also the operator 
multiplication by $V$.

Another possible way to prove the existence of the scattering 
solution is via Fredholm theory, which consists in choosing $u_{sc}$ as the solution of 
\begin{equation}
\left\{
\begin{aligned}
& (I - (\Delta + \lambda \pm i 0)^{-1} \circ V)u_{\rm sc} = (\Delta + \lambda \pm i 0)^{-1}( V u_{\rm in}) \qquad \textnormal{in}\, \R^d,\\
& u_{\rm sc} \, \textnormal{satisfying SRC},
\end{aligned}
\right.
\label{pb:Fredholm}
\end{equation}
where $I$ stands for the identity operator. In order to solve the 
problem \eqref{pb:Fredholm} using the Fredholm alternative, one needs 
to ensure that $(\Delta + \lambda \pm i 0)^{-1} \circ V$ is compact 
in the space where the solutions $u_{\rm sc}$ will belong to, and 
zero is the only solution to the homogeneous counterpart of the 
problem \eqref{pb:Fredholm}.

To apply any of these two schemes one needs appropriate estimates for the resolvent $(\Delta + \lambda \pm i 0)^{-1}$ according to the character or behaviour of $V$. For example, the well-known resolvent estimate ---due to Agmon \cite{zbMATH03493415}---
\begin{equation}
\lambda^{1/2} \| (\Delta + \lambda \pm i 0)^{-1} f \|_{L^{2,-\delta}} \lesssim \| f \|_{L^{2,\delta}},
\label{in:A}
\end{equation}
with $\delta > 1/2$ and $\| f \|_{L^{2,\pm\delta} (\R^d)}^2 = \int_{\R^d} (1 + |x|^2)^{\pm\delta} |f(x)|^2 \, \dd x $, makes possible to prove that
\[ \| (\Delta + \lambda \pm i 0)^{-1} \circ V \|_{\mathcal{L}(L^{2,-\delta}(\R^d))} \lesssim \lambda^{-1/2}\]
with $V \in L^\infty(\R^d)$ and compactly supported. An 
improved version of Agmon's inequality is the following one ---due to
Agmon and H\"ormander \cite{MR0466902}---
\begin{equation}
\lambda^{1/2} \sup_{j\in\N_0}\big(2^{-j/2}\| (\Delta + \lambda \pm i 0)^{-1} f \|_{L^2(D_j)}\big) \lesssim \sum_{j\in \N_0}2^{j/2}\| f \|_{L^2(D_j)},
\label{in:AH-KPV}
\end{equation}
where $\N_0 = \N \cup \{ 0 \}$, $D_j=\{x\in \R^d : 2^{j-1}<|x|\leq 2^j\}$ for $j \in \N$ and $D_0 = \{ x \in \R^d : |x| \leq 1 \}$.
It is very common to let the norm on the left-hand side be 
denoted by $\vvvert \centerdot \vvvert_\ast$ and the one on the right by
$\vvvert \centerdot \vvvert$. Thus,
\[ \vvvert f \vvvert_\ast = \sup_{j\in\N_0}\big(2^{-j/2}\| f \|_{L^2(D_j)}\big), \qquad \vvvert f \vvvert = \sum_{j\in \N_0}2^{j/2}\| f \|_{L^2(D_j)}. \]
Another important and very celebrated resolvent 
estimate is the following ---due to Kenig, Ruiz and Sogge 
\cite{zbMATH04050093}---
\begin{equation}
\| (\Delta + \lambda \pm i 0)^{-1} f \|_{L^p} \lesssim \lambda^{\frac{d}{2}\big( \frac{1}{p^\prime} - \frac{1}{p} \big) - 1} \| f \|_{L^{p^\prime}},
\label{in:KRS}
\end{equation}
where
$ 2/(d+1) \leq 1/p^\prime - 1/p \leq 2/d $, $1/p + 1/p^\prime = 1$ 
and $d \geq 3$. The inequality \eqref{in:KRS} can be used to show 
that, for the range $ 2/(d+1) \leq 1/p < 2/d $, the inequality
\[ \| (\Delta + \lambda \pm i 0)^{-1} \circ V \|_{\mathcal{L}(L^{2p^\prime})} \lesssim \lambda^{\frac{d}{2p} - 1}\]
holds with $1/p^\prime + 1/p = 1$ and $V \in L^q(\R^d)$ compactly 
supported, where $q>d/2$ and $d \geq 3$. The end-point case
$V \in L^{d/2}(\R^d)$ does not follow directly from either the Neumann series argument ---unless there is smallness for
$\| V \|_{L^{d/2}}$--- or the Fredholm alternative. The Neumann series argument fails in the end-point because we only have
\[ \| (\Delta + \lambda \pm i 0)^{-1} \circ V \|_{\mathcal{L}(L^{p_d} )} \lesssim 1,\]
where $p_d$ is the index of the $\dot{H}^1(\R^d)$
Hardy--Littlewood--Sobolev embedding,
$1/p_d = 1/2 - 1/d$. 
The Fredholm theory does not seem to apply for the lack of the 
compactness, specially because $H^1(B_0)$ is not compactly embedded in $L^{p_d}(B_0)$. However, Lavine and Nachman managed to modify the 
procedure, with a formulation that reminds the operator used to prove 
the Birman--Schwinger principle, in order to reach the end-point. To 
make their argument work one needs to use the inequalities 
\eqref{in:A} and \eqref{in:KRS}. We learnt it from 
\cite{zbMATH06156446}. Another improvement of Agmon's inequality is
\begin{equation}
\| (\Delta + \lambda \pm i 0)^{-1} f \|_{H^{s,-\delta}_\lambda} \lesssim \lambda^{-(1/2 - s)} \| f \|_{H^{-s,\delta}_\lambda},
\label{in:CHL}
\end{equation}
with $0 \leq s \leq 1/2$, $\delta > 1/2$ and
\[ \| f \|_{H^{\pm s, \mp \delta}_\lambda } = \| (\lambda - \Delta)^{\pm s/2} f \|_{L^{2,\mp \delta}},\]
where $(\lambda - \Delta)^{\pm s/2}$ stands for 
the multiplier with symbol $(\lambda + |\xi|^2)^{\pm s/2}$. This 
inequality was proved in \cite{doi:10.1142/S0219530519500015} to 
study scattering in the presence of a class of Gaussian random potentials called microlocally isotropic.
The realizations of such potentials are 
compactly supported and belong to the potential Sobolev spaces 
$L^p_{-s}(\R^d)$ with $0 < s \leq 1/2$ and $d/s \leq p < \infty$, 
almost surely.
Recall that the potential Sobolev space $L^p_s(\R^d)$ with $1 < p < \infty$ and $s \in \R$ is defined by $(I - \Delta)^{-s/2} L^p(\R^d)$ with $(I - \Delta)^{-s/2}$ the Bessel potential with symbol $(1 + |\xi|^2)^{-s/2}$.
From \eqref{in:CHL}, Caro, Helin and Lassas showed in 
\cite{doi:10.1142/S0219530519500015} that, for compactly supported 
potentials in $L^p_{-s}(\R^d)$ with $0 < s \leq 1/2$ and $d/s \leq p < \infty$, one has
\[\| (\Delta + \lambda \pm i 0)^{-1} \circ V \|_{\mathcal{L}(H^{s,-\delta}_\lambda)} \lesssim o( \lambda^{-(1/2 - s)} ).\]
The inequality \eqref{in:CHL} can be easily extended to the range
$ 0 \leq s \leq 1$ 
and then used to 
prove ---by the Neumann series argument--- the existence of scattering solution for potentials as 
\eqref{id:V} with $V^0 \in L^\infty(\R^d; \R)$ and $\| \alpha \|_{L^\infty(\Gamma)}$ 
small enough. Despite, we do not know any reference dealing with this 
problem for every dimension $d \geq 2$, we believe that the truth 
challenge of the scattering theory arises when considering potentials 
that are the combination of critically-singular and $\delta$-shell 
potentials. For such potentials, neither the inequality 
\eqref{in:CHL} ---for the full range $0\leq s\leq 1$--- nor \eqref{in:KRS} with no adjustment 
seem to be enough to develop the scattering theory. On the other 
hand, because of the nature of the term $\alpha \, \dd \sigma$, the 
Lavine--Nachman argument might not be easily adapted for potentials 
of the form in \eqref{id:V}. In fact, in this article we develop an 
alternative path that we explain in the next lines.

The approach we propose is inspired by the most recent works studying the 
Calder\'on problem for dimension $d\geq 3$.
This inverse problem consists in determining the electric 
conductivity of a medium from its corresponding Dirichlet-to-Neumann map. The 
key ingredient in the resolution of this problem is a type of solutions 
called complex geometrical optics (CGO for short), first constructed by 
Sylvester and Uhlmann \cite{zbMATH04015323}. Most of the progresses related 
to this problem have consisted in refining the construction of the CGO 
solutions, which boils down to inverting the conjugated Laplacian
$\Delta + 2 \tau \partial_{x_d} + \tau^2$ for at least
$\tau \geq \tau_0 > 0$. In \cite{zbMATH06145493}, Haberman and Tataru 
introduced a family of Bourgain spaces ---denoted here by $\dot{Y}_\tau^s$ with $s \in \R$--- adapted to this 
differential operator\footnote{Actually, the differential operator 
that Haberman and Tataru considered was
$\Delta + 2 \zeta \cdot \nabla$ with $\zeta = \Re \zeta + i \Im \zeta \in \C^d$ so that
$\zeta \cdot \zeta = 0$, and consequently the family of Bourgain spaces they 
introduced, denoted in their work by  $\dot{X}_\zeta^s$, had 
similar norms to $\dot{Y}_\tau^s$ but with
$p_\zeta(\xi) = - |\xi|^2 + i 2 \zeta \cdot \xi$ instead of $q_\tau$. 
Note that $\Delta + 2 \zeta \cdot \nabla = e^{-i\Im\zeta \cdot x} (\Delta + 2 \Re\zeta \cdot \nabla + |\Re \zeta|^2) \circ e^{i \Im\zeta \cdot x}$ and consequently, if $\Re\zeta = \tau T e_d$ with $T \in SO(d)$, then $ q_\tau (\xi) = p_\zeta (T\xi - \Im\zeta)$.}, whose norms were of the form
\[\| f \|_{\dot{Y}^s_\tau} = \| |q_\tau|^s \widehat{f} \|_{L^2},\]
where $q_\tau (\xi) = - |\xi|^2 + i 2\tau \xi_d + \tau^2$ stands for 
the symbol of the conjugated Laplacian. This family of spaces is 
very convenient for several reasons: the first one is because the 
inverse of the conjugated Laplacian is an isometry
\begin{equation}
\| (\Delta + 2 \tau \partial_{x_d} + \tau^2)^{-1} f \|_{\dot{Y}^s_\tau} = \| f \|_{\dot{Y}^{s-1}_\tau}.
\label{id:isometry}
\end{equation}
The regularity of $V$ in \eqref{id:V} make the index $s = 1/2$ play a 
relevant role. The second reason is that, when functions are 
localized in space, the norm for $s=1/2$ controls the $L^2$ norm of 
such functions with a gain of $\tau^{1/2}$. This fact was shown in 
\cite{zbMATH06145493}:
\begin{equation}
\tau^{1/2} \| \chi f \|_{L^2} \lesssim \| f \|_{\dot{Y}^{1/2}_\tau}
\label{in:HT_localization}
\end{equation}
where $\chi \in \mathcal{S}(\R^d)$. Another reason that makes 
relevant this space is the following embedding ---due to Haberman 
\cite{zbMATH06490961}---
\begin{equation}
\| f \|_{L^{p_d}} \lesssim \| f \|_{\dot{Y}^{1/2}_\tau}.
\label{in:H}
\end{equation}
As a consequence of \eqref{in:H} and 
\eqref{id:isometry}, one can derive the inequality
\begin{equation}
\| (\Delta + 2 \tau \partial_{x_d} + \tau^2)^{-1} f \|_{L^{p_d}} \lesssim \| f \|_{L^{p_d^\prime}}
\label{in:KRS_CGO}
\end{equation}
for $d \geq 3$ with $1/p_d + 1/p_d^\prime = 1$. The inequality 
\eqref{in:KRS_CGO} was proved by Kenig, Ruiz and Sogge \cite{zbMATH04050093} 
as a consequence of \eqref{in:KRS} for $1/p^\prime - 1/p = 2/d$, however, 
this was written in the form of a Carleman estimate.

Our strategy in this article is to 
introduce two spaces $X_\lambda$ and $X_\lambda^\ast$ adapted to the 
resolvent operator $(\Delta + \lambda \pm i 0)^{-1}$ for which analogues of 
\eqref{id:isometry}, \eqref{in:HT_localization} and \eqref{in:H} hold. In 
fact, we will see the resolvent estimate
\begin{equation}
\| (\Delta + \lambda \pm i 0)^{-1} f \|_{X_\lambda^\ast} \lesssim \| f \|_{X_\lambda}
\label{in:resolvent}
\end{equation}
and the embedding
\begin{equation}
\lambda^{1/4} \vvvert f \vvvert_\ast + \lambda^{\frac{d}{2}(\frac{1}{p}-\frac{1}{p_d})}\|f\|_{L^p} \lesssim \| f \|_{X_\lambda^\ast},
\label{in:embedding}
\end{equation}
where $p\in[q_d, p_d]$ with $q_d$ so that $2/q_d=(d-1)/(d+1)$ ---the 
index\footnote{For some computations, it is useful to note that $1/q_d = 1/2 - 1/(d+1)$, that is $q_d = p_{d+1}$} in the extension form of the Tomas--Stein theorem.
From the inequalities \eqref{in:resolvent} and \eqref{in:embedding}, 
one can prove that $(\Delta + \lambda \pm i 0)^{-1} \circ V$ is a 
contraction on $X_\lambda^\ast$ for $\lambda \geq \lambda_0 > 0$ 
under a smallness assumption on $\alpha$. This would allow to 
construct the scattering solution $u_{\rm sc}$ by the Neumann series 
argument. In order to avoid assuming 
smallness for $\alpha$, we adopt a strategy that combines the Neumann 
series argument and the Fredholm alternative. First, we use the 
Neumann series argument to construct the resolvent
$(\Delta + \lambda \pm i 0 - V^0)^{-1}$ and prove its boundedness from $X_\lambda$ to $X_\lambda^\ast$. Then, we use the Fredhlom theory to solve the problem
\begin{equation*}
\left\{
\begin{aligned}
& \big(I - (\Delta + \lambda \pm i 0 - V^0)^{-1} \circ (\alpha\, \dd \sigma)\big)u_{\rm sc} = (\Delta + \lambda \pm i 0 - V^0)^{-1}( V u_{\rm in}) \quad \textnormal{in}\, \R^d,\\
& u_{\rm sc} \, \textnormal{satisfying SRC}.
\end{aligned}
\right.
\end{equation*}
Two ingredients are required to apply the Fredholm theory. The first one is the compactness from $X_\lambda^\ast$ to $X_\lambda$ of the operator multiplication by $\alpha\,\dd\sigma$. The second ingredient is a unique continuation property for an equation with a potential as in \eqref{id:V}. Here, we derive this unique continuation using a Carleman estimate that Caro and Rogers proved in \cite{zbMATH06534426} for the Bourgain spaces.

Intuitively, the elements of $X_\lambda^\ast$ should be thought as 
functions with some integrability whose weak (up to first order) derivatives have also certain (but different) integrability properties.
In fact,
$$X^\ast_\lambda \subset L^p_{d(1/p - 1/p_d)}(\R^d) \cap (I - \Delta)^{-1/2} B_\ast$$
with $p \in [q_d, p_d]$, $L^p_{d(1/p - 1/p_d)}(\R^d)$ the potential Sobolev space with differentiability index $d(1/p - 1/p_d)$ and integrability index $p$, and $B_\ast$ the Banach space defined by the norm $\vvvert \centerdot \vvvert_\ast$ ---this inclusion follows by changing slightly the proofs of the  lemmas \ref{lem:Lpembedding} and \ref{lem:L2embedding} in the section \ref{sec:resolvent}.
Contrarily, the elements of $X_\lambda$ are actual distributions, an example of them are elements of
$$L^{p^\prime}_{-d(1/p_d^\prime - 1/p^\prime)}(\R^d) + (I - \Delta)^{1/2} B$$
with $p^\prime$ and $p_d^\prime$ the dual exponents of $p$ and $p_d$, and $B$ the Banach space defined by the norm $\vvvert \centerdot \vvvert$. Actually, the latter space is included in $X_\lambda$.
Despite the nature of the spaces $X_\lambda^\ast$ and $X_\lambda$, the inequality \eqref{in:resolvent} is somehow equivalent to 
a combination of \eqref{in:AH-KPV} and \eqref{in:KRS} (see the 
remarks \ref{rem:Ylambda} and \ref{rem:Zlambda} in section \ref{sec:resolvent}). However, the 
inequality \eqref{in:resolvent} is 
better adapted than \eqref{in:AH-KPV} and \eqref{in:KRS} to deal with 
potentials $V$ as in \eqref{id:V}, in 
this sense our new estimate is a refinement of the classical ones.
The ideal 
situation would be to define the spaces $X_\lambda^\ast$ and
$X_\lambda$ 
through the  $L^2$ norms in the frequency side with the weights $\sqrt{m_\lambda}$ and $1/\sqrt{m_\lambda}$
respectively, where $m_\lambda(\xi) = |\lambda-|\xi|^2|$. However, it 
is not as straightforward as this since $1/\sqrt{m_\lambda}$
is not in 
$L^2_{\textnormal{loc}}(\R^d)$ ---see how we overcome this issue in the 
definitions \ref{def:Ylambda}, \ref{def:Zlambda} and 
\ref{def:Xlambda} in the section \ref{sec:scattering}.

Our approach provides a suitable framework to construct the 
scattering solution using a strategy that combines the Neumann series and the Fredholm alternative. Given $y \in \R^d$, consider $u_{\rm in}^\pm(x, y) = \Phi_\lambda^\pm (y - x) $ with
\[\Phi^\pm_\lambda = (\Delta + \lambda \pm i0)^{-1} \delta_0,\]
where $\delta_0$ is the Dirac mass at $0$.

\begin{main-theorem}\label{th:scattering} \sl Consider $d \geq 3$. There exists a  positive $\lambda_0 = \lambda_0 (d, V^0, R_0)$ so that, 
we can find $u_{\rm sc}^\pm (\centerdot, y) \in X_\lambda^\ast$ solving the problem
\begin{equation*}
\left\{
\begin{aligned}
& (\Delta + \lambda - V)u_{\rm sc}^\pm(\centerdot, y) = V u_{\rm in}^\pm(\centerdot, y) & & \textnormal{in}\, \R^d,\\
& \sup_{|x|=R} \Big| \frac{x}{|x|} \cdot \nabla u_{\rm sc}^\pm (x, y) \mp i \lambda^{1/2} u_{\rm sc}^\pm (x, y)  \Big| = o_y(R^{-\frac{d-1}{2}}) & & R \geq R_0,
\end{aligned}
\right.
\end{equation*}
for every $\lambda \geq \lambda_0$ and $y \in \R^d \setminus B_0$.
Moreover, $u_{\rm sc}^\pm (\centerdot, y)$ is the only solution of the previous problem.
\end{main-theorem}

\begin{remark} For dimension $d=2$, we could have used the Neumann series argument and our estimates to state that there exist $\varepsilon = \varepsilon(d, \Gamma, R_0)$ and $\lambda = \lambda(d, V^0, R_0)$ so that, if $\| \alpha \|_{L^\infty(\Gamma)} \leq \varepsilon$, then there exists a unique scattering solution $u_{\rm sc}^\pm (\centerdot, y) \in X_\lambda^\ast$ for every $\lambda \geq \lambda_0$. We have not combined the Neumann series argument and the Fredholm alternative in this situation because we have not found an appropriate unique continuation for a potential $V$ as in \eqref{id:V} for $d=2$.
\end{remark}

\subsection{Inverse scattering}
The inverse point-source problem we study in this paper consists 
in determining a potential $V$ as in \eqref{id:V} from the knowledge of $u_{\rm sc}^\pm|_{\partial B_0 \times \partial B_0}$ for a fixed energy $\lambda$, where $u_{\rm sc}^\pm(\centerdot, y)$ is the scattering solution of the theorem \ref{th:scattering} yielded by the incident wave $u_{\rm in}^\pm(\centerdot, y) = \Phi_\lambda^\pm (y - \centerdot) $.
%

\begin{main-theorem}\label{th:uniqueness} \sl Consider $d \geq 3$. Let $V_1 = V^0_1 + \alpha_1 \, \dd \sigma_1 $ and $V_2 = V^0_2 + \alpha_2 \, \dd \sigma_2$ be two electric potentials as in \eqref{id:V}, where $\sigma_j$ is the surface measure of
$\Gamma_j$. Let $\lambda_0 = \lambda_0 (d, V^0_1, V^0_2, R_0)$ be so that the scattering solutions $u^\pm_{{\rm sc},1} (\centerdot, y)$ and $u^\pm_{{\rm sc},2} (\centerdot, y)$ of the theorem \ref{th:scattering} with potentials $V_1$ and $V_2$ are available for every $\lambda \geq \lambda_0$. Then, 
\[u^\pm_{{\rm sc},1}|_{\partial B_0 \times \partial B_0} = u^\pm_{{\rm sc},2}|_{\partial B_0 \times \partial B_0} \, \textsl{for a fixed}\, \lambda \geq \lambda_0 \,\Longrightarrow \, V_1 = V_2.\]
\end{main-theorem}

\begin{remark}\sl Note that $V_1 = V_2$ implies that $V_1^0 = V_2^0$, 
$\Gamma_1 = \Gamma_2$ and $\alpha_1 = \alpha_2$. Indeed, we can test 
$V_1 - V_2$ with a sequence of functions $\phi_n$ that concentrates around $\Gamma_1 \cup \Gamma_2$ so that
$\int_{\R^d} (V_1^0 - V_2^0) \phi_n $ vanishes as the functions concentrates around this set of measure zero. This implies that
$\Gamma_1 = \Gamma_2$ and consequently that $\alpha_1 = \alpha_2$. At this point obviously $V^0_1 = V^0_2$.
\end{remark}

To address the question of uniqueness for this fixed-energy inverse scattering problem, we adopt the approach that
H\"ahner and Hohage followed in \cite{zbMATH01670370} to prove some 
stability estimates for a similar problem for the 
acoustic equation. We start by proving an \textit{orthogonality} 
relation in the spirit of Alessandrini's identity for the Calder\'on 
problem, that is,
\begin{equation}
\langle (V_1 - V_2) v_1, v_2 \rangle = 0
\label{id:integral}
\end{equation}
for all $v_j$ solution of $(\Delta + \lambda - V_j)v_j = 0$ in $B_0$.
Then, we construct CGO solutions ---as 
Sylvester and Uhlmann did in \cite{zbMATH04015323}--- in the 
form
\[v_j(x) = e^{\zeta_j \cdot x} (1 + w_j(x)),\]
where $\zeta_j \in \C^d$ so that $\zeta_j \cdot \zeta_j = - \lambda$ 
and $\zeta_1 + \zeta_2 = -i\kappa$ for an arbitrarily given
$\kappa \in \R^d$ ---which is possible in dimension $d \geq 3$---, and the correction term $w_j$ vanishes in a 
specific sense when $|\zeta_j|$ grows. Because of the $\delta$-shell components $\alpha_1 \, \dd \sigma_1$ and $\alpha_2 \, \dd \sigma_2$
of the potentials $V_1$ and $V_2$, we follow the ideas introduced by 
Haberman and Tataru in \cite{zbMATH06145493} in order to ensure the 
asymptotic behaviour of $w_j$ when $|\zeta_j|$ grows. However, 
since no smallness is assumed for $\alpha_j$, we also  
require at this stage the Carleman estimate proved by Caro and Rogers 
in \cite{zbMATH06534426}. The critically-singular components $V^0_1$ 
and $V^0_2$ can be treated thanks to the embedding \eqref{in:H} due 
to Haberman \cite{zbMATH06490961}. Finally, we 
plug in the CGO solutions to \eqref{id:integral} and make $|\zeta_1|$ 
and $|\zeta_2|$ grow. Thus, we can conclude that the Fourier transform of $V_1 - V_2$ is identically zero, that is,
\[\mathcal{F} \big(V_1 - V_2 \big)(\kappa) = 0, \qquad \forall \kappa \in \R^d.\]
The injectivity of the Fourier transform allows us to conclude that $V_1 = V_2$.

\subsection{Some previous results}
The spaces
$$L^p_{d(1/q_d - 1/p_d)}(\R^d) \cap (I - \Delta)^{-1/2} B_\ast \quad \textnormal{and} \quad L^{p^\prime}_{-d(1/p_d^\prime - 1/q_d^\prime)}(\R^d) + (I - \Delta)^{1/2} B$$
are the spaces chosen by Ionescu and Schlag in \cite{zbMATH05014647} to prove the limit absorption principle for a large class of perturbations. It turns out that their basic estimate ---with an explicit control in $\lambda$--- can be derived from \eqref{in:resolvent} and the relation of these spaces with $X_\lambda^\ast$ and $X_\lambda$. Another resolvent estimate that seems to follow from ours, after an adjustment in the norm of $X_\lambda^\ast$, is the one due to Ruiz and Vega ---Theorem 1.2 in \cite{zbMATH00055931}. See also the work of Goldberg and Schlag \cite{zbMATH02138311}.

Regarding previous results on inverse scattering with $\delta$-shell potentials, see the work of Mantile, Posilicano and Sini \cite{zbMATH06933409} in dimension $d = 3$. The point source-scattering have been previously studied in \cite{zbMATH01670370} by H\"ahner and Hohage in acoustic media, and by Ola and Somersalo \cite{zbMATH00926323} for Maxwell equations.

The literature on inverse scattering is rather wide and we cite only a few works where the measurements are assumed to be modelled by the far-field pattern. Colton and Kirsch introduced in \cite{zbMATH00939857} the linear sampling method to determine the support of an imperfect conductor. Uniqueness and reconstruction for the inverse scatteing problem in an acoustic medium was proved by Nachman \cite{zbMATH04105476}, Novikov \cite{zbMATH04129351}, and Ramm \cite{zbMATH04028038}. The stability question was first addressed by Stefanov \cite{zbMATH00001021}, and then improved by H\"ahner and Hohage \cite{zbMATH01670370}.

\subsection{The outline of the paper}
The section \ref{sec:scattering} is 
devoted to the study of the direct scattering from a point source. We first 
pose rigorously the point-source scattering problem. Then, we
introduce the spaces $X_\lambda$ and $X_\lambda^\ast$, and state rigorously the inequalities \eqref{in:resolvent} and  \eqref{in:embedding}.
Afterwards, we construct the resolvent $(\Delta + \lambda \pm i0 - V^0)^{-1}$ by a Neumann series argument and then we use the Fredholm alternative to prove the existence of the scattering solution.
The inverse problem is considered in the section 
\ref{sec:inverse_scattering}.
First, we prove a couple of lemmas that are required for the orthogonality identity \eqref{id:integral}. Then, we construct the CGO solutions and show the uniqueness of the potentials. 
In the section \ref{sec:resolvent}, we first state a couple of refined resolvent 
estimates in the spirit of \eqref{in:resolvent}.
There, we also provide a rather simple proofs of \eqref{in:AH-KPV} and 
\eqref{in:KRS}. We find specially interesting the proof of \eqref{in:KRS}, 
where we do not use Stein's interpolation theorem and reach the endpoint in the
case $d = 2$. The last part of the section \ref{sec:resolvent} contains 
some connections of our refined resolvent estimates with the estimates that 
Sylvester and Uhlmann used to construct the CGO solutions, as well as, the 
inequalities \eqref{in:HT_localization} and \eqref{in:H} 
proved by Haberman and Tataru. The article finishes with 
an appendix where
we address the most basic questions of the functional spaces $X_\lambda^\ast$ and $X_\lambda$.

The section \ref{sec:resolvent} may be read independently of the sections \ref{sec:scattering} and \ref{sec:inverse_scattering}, only some notations and definitions from the previous sections would be required. However, the sections \ref{sec:scattering} and \ref{sec:inverse_scattering} are full of references and calls to the section \ref{sec:resolvent}. Thus, if readers choose to follow the order proposed by the authors, they would get a global picture of the direct and inverse problems from the sections \ref{sec:scattering} and \ref{sec:inverse_scattering} postponing the details for the section \ref{sec:resolvent}.

\begin{acknowledgements} The authors thank Alberto Ruiz for his valuable comments. The authors were partially funded by BERC 2018-2021 and Severo Ochoa SEV-2017-0718. Additionally, PC is also funded by Ikerbasque and PGC2018-094528-B-I00, and AG by Juan de la Cierva fellowship IJCI-2015-25009.
\end{acknowledgements}

\section{Scattering theory} \label{sec:scattering}
In point-source scattering theory, the incident wave is typically 
chosen as certain fundamental solutions. More precisely, given $y \in \R^d$, the incident wave is given by $u_{\rm in}^\pm(x, y) = \Phi_\lambda^\pm (y - x) $ with
\[\Phi^\pm_\lambda = (\Delta + \lambda \pm i0)^{-1} \delta_0,\]
where $\delta_0$ is the Dirac mass at $0$. The previous identity is understood as
\begin{equation*}
\langle\Phi^\pm_\lambda, f\rangle= \frac{1}{(2\pi)^{d/2}} \Big[ \lim_{\epsilon\to 0}\int_{m_\lambda>\epsilon}\frac{\widehat{f}(\xi)}{\lambda-|\xi|^2} \, \dd \xi \mp i\frac{\pi}{2\lambda^{1/2}}\int_{S_\lambda}\widehat{f}(\xi) \, \dd S_\lambda(\xi) \Big]
\end{equation*}
for every $f\in \mathcal{S}(\R^d)$, where $S_\lambda=\{\xi\in \R^d : |\xi|=\lambda^{1/2}\}$ and $\dd S_\lambda$ stands for the volume form on $S_\lambda$. One can check that $\Phi^\pm_\lambda $ is the fundamental solution solving the problem
\begin{equation*}
\left\{
\begin{aligned}
& \Delta \Phi_\lambda^\pm  + \lambda \Phi_\lambda^\pm = \delta_0 & & \textnormal{in}\, \R^d,\\
& \frac{x}{|x|} \cdot \nabla \Phi_\lambda^\pm (x) \mp i \lambda^{1/2}  \Phi_\lambda^\pm (x) = \mathcal{O} (|x|^{-\frac{d + 1}{2}}) & & \textnormal{for}\, |x|\geq 1.
\end{aligned}
\right.
\end{equation*}
The last condition corresponds to either the ingoing or the outgoing
SRC. Our goal in this section is to construct the scattering solution 
$u_{\rm sc}^\pm(\centerdot, y)$ solving the problem
\begin{equation}
\left\{
\begin{aligned}
& (\Delta + \lambda - V)u_{\rm sc}^\pm(\centerdot, y) = V u_{\rm in}^\pm(\centerdot, y) & & \textnormal{in}\, \R^d,\\
& \sup_{|x|=R} \Big| \frac{x}{|x|} \cdot \nabla u_{\rm sc}^\pm (x, y) \mp i \lambda^{1/2} u_{\rm sc}^\pm (x, y)  \Big| = o_y(R^{-\frac{d-1}{2}}) & & R \geq 1,
\end{aligned}
\right.
\label{pb:scattering}
\end{equation}
with $ y \in \R^d \setminus B_0$ and $V$ as in \eqref{id:V}.

As we mention in the introduction, the scattering solution $u_{\rm sc}(\centerdot, y)$ will 
be constructed in a space $X_\lambda^\ast$. Ideally this space would be defined through 
the symbol $\sqrt{m_\lambda}$, with $m_\lambda(\xi) = |\lambda-|\xi|^2|$, however, this is not possible. If $X_\lambda^\ast$ were defined by the symbol $\sqrt{m_\lambda}$, its
pre-dual $X_\lambda$ would have to be defined by the symbol $1/ \sqrt{m_\lambda}$, which is not locally square-integrable around
the \textit{critical hypersurface} $S_\lambda$. For that reason, 
given $\lambda > 0$, the integer $k_\lambda$ so that
$2^{k_\lambda - 1} < \lambda^{1/2} \leq 2^{k_\lambda} $ will play a 
special role. Thus, to avoid the \textit{critical frequencies} around
$S_\lambda$, we introduce the set
\[I = \{ k_\lambda - 2, k_\lambda - 1, k_\lambda, k_\lambda + 1  \}\]
and use the Littlewood--Paley projectors $P_k$ and $P_{\leq k}$. To define them, it is enough to consider $\phi \in \mathcal{S}(\R^d)$ supported in $\{ \xi \in \R^d: |\xi| \leq 2 \}$ and $\phi (\xi) = 1$ whenever $|\xi|\leq 1$, and the function $\psi(\xi) = \phi(\xi) - \phi(2\xi)$. Then,
$$\widehat{P_k f} (\xi) = \psi(\xi/2^k) \widehat{f}(\xi), \qquad \widehat{P_{\leq k} f} (\xi) = \phi(\xi/2^k) \widehat{f}(\xi).$$
In this paper, the projector $P_{\leq k_\lambda - 3}$ will have a relevant importance, and will be denoted for simplicity by $P_{<I}$

The space $X_\lambda^\ast$ will be introduced as the dual of
$X_\lambda$ which in turn is defined as the sum of two spaces
$Y_\lambda$ and $Z_\lambda$. These later spaces with their 
corresponding duals $Y_\lambda^\ast$ and $Z_\lambda^\ast$ come forth 
to refine the estimates \eqref{in:AH-KPV} and \eqref{in:KRS}, 
respectively.
\begin{definition}\label{def:Ylambda}\sl Let $Y_\lambda$ be the set 
of $f\in\mathcal{S}^\prime(\R^d)$  such that 
\begin{equation*}
\|m_\lambda^{-1/2}\widehat{P_{<I} f}\|_{L^2}^2 + \sum_{k\in I} \lambda^{-1/2}\vvvert P_k f \vvvert^2 +\sum_{k > k_\lambda + 1}\|m_\lambda^{-1/2}\widehat{P_k f}\|_{L^2}^2 <\infty,
\end{equation*}
where $m_\lambda(\xi) = |\lambda-|\xi|^2|$.
For $f \in Y_\lambda$, define the norm
\begin{equation*}
\|f\|_{Y_\lambda}^2 = \|m_\lambda^{-1/2}\widehat{P_{<I} f}\|_{L^2}^2 + \sum_{k\in I} \lambda^{-1/2}\vvvert P_k f \vvvert^2 +\sum_{k > k_\lambda + 1}\|m_\lambda^{-1/2}\widehat{P_k f}\|_{L^2}^2.
\end{equation*}
\end{definition}

To introduce the space $Z_\lambda$, it is convenient to remember that 
$q_d$ is $2/q_d=(d-1)/(d+1)$ for $d\geq 2$, while $p_d$ is
$1/p_d = 1/2 - 1/d$ if $d \geq 3$. In dimension $d = 2$, we write
$ p_2=\infty$.
\begin{definition}\label{def:Zlambda}\sl Let $Z_{\lambda, p^\prime}$ with $p^\prime \in [p_d^\prime, q_d^\prime]$ be the set 
of $g\in\mathcal{S}^\prime(\R^d)$ such that 
\begin{equation*}
\|m_\lambda^{-1/2}\widehat{P_{<I} g}\|_{L^2}^2 + \sum_{k\in I} \lambda^{d(\frac{1}{p^\prime}-\frac{1}{p_d^\prime})}\|P_k g\|_{L^{p^\prime}}^2 +\sum_{k > k_\lambda + 1}\|m_\lambda^{-1/2}\widehat{P_k g}\|_{L^2}^2 <\infty,
\end{equation*}
where $m_\lambda(\xi) = |\lambda-|\xi|^2|$.
For $g \in Z_{\lambda, p^\prime}$, define the norm
\begin{equation*}
\|g\|_{Z_{\lambda, p^\prime}}^2 = \|m_\lambda^{-1/2}\widehat{P_{<I} g}\|_{L^2}^2 + \sum_{k\in I} \lambda^{d(\frac{1}{p^\prime}-\frac{1}{p_d^\prime})}\|P_k g\|_{L^{p^\prime}}^2 +\sum_{k > k_\lambda + 1}\|m_\lambda^{-1/2}\widehat{P_k g}\|_{L^2}^2.
\end{equation*}
Here $q_d^\prime$ and $p_d^\prime$ are the dual exponents of $q_d$ and $p_d$ respectively, in particular, $p_2^\prime = 1$. For simplicity, we write $Z_\lambda$ instead of $Z_{\lambda, q_d^\prime}$.
\end{definition}
\begin{remark} By Bernstein's inequality
\[ \|g\|_{Z_\lambda} \lesssim \|g\|_{Z_{\lambda, p^\prime}} \lesssim \|g\|_{Z_{\lambda, p_d^\prime}},  \]
and therefore,
\[ Z_{\lambda, p_d^\prime} \subset Z_{\lambda, p^\prime} \subset Z_\lambda. \]
\end{remark}

Now, we are in position to state the precise definitions of the 
spaces $X_\lambda$ and $X_\lambda^\ast$.
\begin{definition}\label{def:Xlambda}\sl Let $X_\lambda$ be the set 
of $h \in\mathcal{S}^\prime(\R^d)$ such that $h = f + g$ with
$f \in Y_\lambda$ and $g \in Z_\lambda$. For $h \in X_\lambda$, define the norm
\[\| h \|_{X_\lambda} = \inf \{ \| f \|_{Y_\lambda} + \| g \|_{Z_\lambda} : h = f + g \}.\]
Note that the infimum is taken over all representation $h = f + g$ with $f \in Y_\lambda$ and $g \in Z_\lambda$.

The Banach space $(X_\lambda^\ast, \| \centerdot \|_{X_\lambda^\ast} )$ is defined as the dual space of $(X_\lambda, \| \centerdot \|_{X_\lambda} )$.
\end{definition}

To construct the solutions in this functional analytical framework, these 
spaces have to satisfy some basic properties that are stated below and proved 
in the appendix \ref{app:framework}.

\begin{proposition}\label{prop:density} \sl The Schwartz class $\mathcal{S}(\R^d)$ is dense in $Y_\lambda$ and $Z_\lambda$ with their corresponding norms. In particular, $\mathcal{S}(\R^d)$ is also dense in $X_\lambda$.
\end{proposition}

\begin{proposition} \label{prop:Xlambda}\sl The pair $(X_\lambda, \| \centerdot \|_{X_\lambda} )$ 
is a Banach space. Its norm can be computed testing on duals elements as follows:
\begin{equation}
\|f\|_{X_\lambda} = \sup_{u \in X_\lambda^\ast \setminus \{ 0 \}} \frac{\langle f, u\rangle}{\| u \|_{X_\lambda^\ast}}.
\label{id:characterizationXlambda}
\end{equation}
\end{proposition}

\begin{proposition}\label{prop:Xlambdaast}\sl The space $X_\lambda^\ast$ is isomorphic to the space 
of $u\in\mathcal{S}^\prime(\R^d)$ so that
\begin{equation*}
\|m_\lambda^{1/2}\widehat{P_{<I} u}\|_{L^2}^2 + \sum_{k\in I} \big[ \lambda^\frac{1}{2} \vvvert P_k u \vvvert_\ast^2 + \lambda^{d(\frac{1}{q_d}-\frac{1}{p_d})}\|P_k u\|_{L^{q_d}}^2 \big] +\sum_{k > k_\lambda + 1}\|m_\lambda^{1/2}\widehat{P_k u}\|_{L^2}^2 <\infty,
\end{equation*}
endowed with the norm
\begin{equation}
\begin{aligned}
\bigg( \sum_{k\in I} \big[ \lambda^{1/2} \vvvert P_k u \vvvert_\ast^2 + & \, \lambda^{d(\frac{1}{q_d}-\frac{1}{p_d})} \|P_k u\|_{L^{q_d}}^2 \big] \\
& + \|m_\lambda^{1/2}\widehat{P_{<I} u}\|_{L^2}^2  + \sum_{k > k_\lambda + 1}\|m_\lambda^{1/2}\widehat{P_k u}\|_{L^2}^2 \bigg)^{1/2}.
\end{aligned}
\label{term:Xlambdaast}
\end{equation}Finally, $\mathcal{S}(\R^d)$ is dense in $X_\lambda^\ast$.
\end{proposition}


These spaces have been constructed to make the following 
theorems hold.
\begin{theorem}\label{th:resolvent} \sl There exists a constant $C > 0$ only depending on $d$ so that
\[\| (\Delta + \lambda \pm i 0)^{-1} f \|_{X_\lambda^\ast} \leq C \| f \|_{X_\lambda}\]
for all $f\in X_\lambda$.
\end{theorem}
\begin{proof}
A standard density argument together with the proposition \ref{prop:density} reduces the theorem to prove the inequality for every $f \in \mathcal{S}(\R^d)$. Now, by the proposition \ref{prop:Xlambdaast} and the lemmas \ref{lem:Ylambda} and \ref{lem:Zlambda} ---in the section \ref{sec:refined}--- we obtain that
\[\| (\Delta + \lambda \pm i 0)^{-1} f \|_{X_\lambda^\ast} \lesssim \| f \|_{Y_\lambda} + \| f \|_{Z_\lambda}\]
for all $f \in \mathcal{S}(\R^d)$. Since the left-hand side of the previous inequalities is independent of the representation of $f $ as $ g + h $ with $g \in Y_\lambda$ and $h \in Z_\lambda$ and $f = 1/2 f + 1/2 f $ is one of the possible ones, we just need to take the infimum on the right-hand side to conclude that
\[\| (\Delta + \lambda \pm i 0)^{-1} f \|_{X_\lambda^\ast} \lesssim \| f \|_{X_\lambda}\]
for all $f \in \mathcal{S}(\R^d)$.
\end{proof}

\begin{theorem}\label{th:embedding} \sl Consider $p\in[q_d, p_d]$.  There exists a constant $C > 0$ only depending on $d$ and $p$ so that
\[\lambda^{1/4} \vvvert u \vvvert_\ast + \lambda^{\frac{d}{2}(\frac{1}{p}-\frac{1}{p_d})}\| u \|_{L^p} \leq C \| u \|_{X_\lambda^\ast}\]
for every $u \in \mathcal{S}(\R^d)$.
\end{theorem}
\begin{proof}
This theorem is a consequence of the lemmas \ref{lem:Lpembedding} and 
\ref{lem:L2embedding} ---in the section \ref{sec:classical}--- and the proposition \ref{prop:Xlambdaast}.
\end{proof}
Next, we use the previous embedding to estimate the norm of the operator multiplication by $V^0$.
\begin{corollary}\label{cor:multiplication} \sl There exists a constant $C> 0$ that only depends on $d$ and $R_0$ so that
\[\| V^0\|_{\mathcal{L}(X_\lambda^\ast; X_\lambda)} \leq C \big(\lambda^{-1/4} + \|\mathbf{1}_F V^0\|_{L^{d/2}}\big),\]
where $F = \{x\in\R^d : |V^0(x)|> \lambda^{1/4}\}$.
\end{corollary}
\begin{proof}
We use \eqref{id:characterizationXlambda} in the proposition \ref{prop:Xlambda} to estimate $\| V^0 \|_{\mathcal{L} (X_\lambda^\ast; X_\lambda)}$. Start by writing
\begin{equation}\label{eq:potential}
\langle V^0 f, g\rangle=\int_{\R^d} V^0 fg
\end{equation}
with $f$ and $g$ in $\mathcal{S}(\R^d)$.
Since the support of $V$ is contained in $B_0$, the support of $V^0$ 
is also contained in $B_0$. Then, $f$ and $g$ in \eqref{eq:potential} 
can be replaced by $\chi f$ and $\chi g$ with $\chi$ a smooth cut-off 
function supported in $2B_0$ and so that $\chi(x) = 1$ for all
$x \in B_0$. Thus,
$$
\int_{\R^d} V^0 f g \, = \int_{\R^d}\mathbf{1}_E V^0 \chi f \chi g\, + \int_{\R^d}\mathbf{1}_F V^0 f g\,,
$$
where $E=\{x\in\R^d : |V^0(x)|\leq M\}$, $F=\R^d\setminus E$, and
$\mathbf{1}_E$ and $\mathbf{1}_F$ stand for the characteristic 
functions of $E$ and $F$. Using H\"older's inequality, we obtain
\begin{equation}
\begin{aligned}
\Big |\int_{\R^d} V^0 \chi f \chi g \, \Big|& \leq M\|\chi f\|_{L^2}\|\chi g\|_{L^2}+\|\mathbf{1}_F V^0\|_{L^{d/2}}\|f\|_{L^{p_d}}\|g\|_{L^{p_d}}\\
&\lesssim M\vvvert f \vvvert_\ast \vvvert g \vvvert_\ast+\|\mathbf{1}_F V^0\|_{L^{d/2}}\|f\|_{L^{p_d}}\|g\|_{L^{p_d}} \\
&\lesssim M\lambda^{-1/2}\|f\|_{X_\lambda^\ast}\|g\|_{X_\lambda^\ast}+\|\mathbf{1}_F V^0\|_{L^{d/2}}\|f\|_{X_\lambda^\ast}\|g\|_{X_\lambda^\ast}.
\end{aligned}
\label{in:V0}
\end{equation}
In the last inequality we have used the theorem \ref{th:embedding}.
From the inequalities \eqref{in:V0} together 
with the density of $\mathcal{S}(\R^d)$ in $X_\lambda^\ast$ provided 
by the proposition \ref{prop:Xlambdaast}, we conclude the statement of the corollary by choosing $M = \lambda^{1/4}$.
\end{proof}

As a direct consequence of the theorem \ref{th:resolvent} and the 
corollary \ref{cor:multiplication} we can estimate $\|(\Delta + \lambda \pm i 0)^{-1} \circ V^0 \|_{\mathcal{L}(X_\lambda^\ast)}$.

\begin{corollary}\label{cor:iteration} \sl There exists a positive $\lambda_0 = \lambda_0 (d, V^0, R_0)$ so that
\[\|(\Delta + \lambda \pm i 0)^{-1} \circ V^0 \|_{\mathcal{L}(X_\lambda^\ast)} < 1\]
for all $\lambda \geq \lambda_0$.
\end{corollary}
\begin{proof} Applying the theorem \ref{th:resolvent} and the 
corollary \ref{cor:multiplication} and noting that $\|\mathbf{1}_F V^0\|_{L^{d/2}}$ tends to $0$ as $\lambda$ grows, we check that the statement holds.
\end{proof}
This corollary is the basic ingredient to perform the Neumann series 
argument sketched in the introduction. In fact, by the corollary \ref{cor:iteration} we have that the series
\begin{equation}
\sum_{n \in \N} [(\Delta + \lambda \pm i 0)^{-1} \circ V^0]^{n-1} (u)
\label{term:series}
\end{equation}
converges in $X_\lambda^\ast$, for every $u \in X_\lambda^\ast$. Thus, we can construct the resolvent
$$(\Delta + \lambda \pm i 0 - V^0)^{-1}$$
and prove its boundedness from $X_\lambda$ to $X_\lambda^\ast$.
\begin{proposition}\label{prop:resolventV0} \sl The operator defined by
\[(\Delta + \lambda \pm i 0 - V^0)^{-1} f = \sum_{n \in \N} [(\Delta + \lambda \pm i 0)^{-1} \circ V^0]^{n-1} ((\Delta + \lambda \pm i 0)^{-1} f), \]
for every $f \in X_\lambda$,
is bounded from $X_\lambda$ to $X_\lambda^\ast$. Moreover, $u^\pm = (\Delta + \lambda \pm i 0 - V^0)^{-1} f $ solves the equation
\begin{equation}
(\Delta + \lambda - V^0) u^\pm = f  \, \, \textnormal{in}\, \R^d,
\label{eq:equation}
\end{equation}
and, if $f$ is compactly supported in $B_0$, then $u^\pm $ satisfies 
the Sommerfeld radiation condition
\[\sup_{|x|=R} \Big| \frac{x}{|x|} \cdot \nabla u^\pm (x) \mp i \lambda^{1/2} u^\pm (x)  \Big| = o(R^{-\frac{d-1}{2}}) \]
for all $R \geq R_0$.
\end{proposition}
\begin{proof}
The fact that $(\Delta + \lambda \pm i 0 - V^0)^{-1}$ is well-defined 
in $X_\lambda$ follows from the convergence of the series 
\eqref{term:series}, which is consequence of the corollary 
\ref{cor:iteration}. The boundedness from $X_\lambda$ to
$X_\lambda^\ast$ follows from the theorem \ref{th:resolvent} and the 
fact that the series \eqref{term:series} defines a bounded operator 
in $X_\lambda^\ast$.
To check that $u^\pm$ solves \eqref{eq:equation} we just need to note that
\begin{equation}
\begin{aligned}
u^\pm &= (\Delta + \lambda \pm i 0)^{-1} f + \sum_{n \in \N} [(\Delta + \lambda \pm i 0)^{-1} \circ V^0]^n ((\Delta + \lambda \pm i 0)^{-1} f)\\
&= (\Delta + \lambda \pm i 0)^{-1} f + (\Delta + \lambda \pm i 0)^{-1} (V^0 u^\pm).
\end{aligned}
\label{id:upm}
\end{equation}
Last identity holds by the corollary \ref{cor:iteration}. Thus, testing the differential operator $\Delta + \lambda$ with $u^\pm$ and using the identity \eqref{id:upm}, we obtain that
\[(\Delta + \lambda - V^0)u^\pm = f \, \, \textnormal{in}\, \R^d.\]
To finish the proof of this proposition, we need to check that $u^\pm$ satisfies the corresponding radiation condition. Start by noting that
\[ u^\pm = (\Delta + \lambda \pm i 0)^{-1}  \sum_{n \in \N} [V^0 \circ (\Delta + \lambda \pm i 0)^{-1}]^{n-1} (f). \]
To justify this identity, we use the boundedness of $(\Delta + \lambda \pm i 0)^{-1}$ from $X_\lambda$ to $X_\lambda^\ast$ and that, for every $\lambda \geq \lambda_0$,
\[\| V^0 \circ (\Delta + \lambda \pm i 0)^{-1} \|_{\mathcal{L} (X_\lambda)} < 1.\]
The contraction of $V^0 \circ (\Delta + \lambda \pm i 0)^{-1}$ in $X_\lambda$ is a consequence of
the corollary \ref{cor:multiplication} and the theorem \ref{th:resolvent}. Note that $u^\pm = (\Delta + \lambda \pm i 0)^{-1} g $, with 
\[ g = \sum_{n \in \N} [V^0 \circ (\Delta + \lambda \pm i 0)^{-1}]^{n-1} (f) \in X_\lambda \]
and compactly supported in $B_0$. Since $g \in X_\lambda$, one can check that $u^\pm$ satisfies the equation
$(\Delta + \lambda) u^\pm = g$.
By Theorem 11.1.1 in \cite{zbMATH02123716}, we have that the restriction of $u^\pm$ to $\R^d \setminus \supp g$ is smooth. On the other hand, since $g$ is compactly supported and the function $y \mapsto \Phi^\pm_\lambda (x - y)$ is smooth in any open neighbourhood $N_g$ of $\supp g$, for every $x \in \R^d \setminus \overline{N_g}$ then,
\[\langle u^\pm , \phi \rangle = \langle g, \int_{\R^d}  \Phi^\pm_\lambda (x - \centerdot) \phi(x) \, \dd x \rangle  = \int_{\R^d} \langle g, \Phi^\pm_\lambda (x - \centerdot) \rangle \phi(x) \, \dd x \]
for all $\phi \in \mathcal{S}(\R^d)$ with $\supp \phi \subset \R^d \setminus \supp g$. Then, the representation formula 
\begin{equation*}
u^\pm(x) = \langle g, \Phi^\pm_\lambda (x - \centerdot) \rangle, \quad \forall x \in \R^d \setminus \supp g
\end{equation*}
holds. To check the radiation condition, we proceed as follows
\[\Big| \frac{x}{|x|} \cdot \nabla u^\pm (x) \mp i \lambda^{1/2} u^\pm (x)  \Big| \leq \|g \|_{X_\lambda} \Big\| \chi \Big( \frac{x}{|x|} \cdot \nabla_x \mp i \lambda^{1/2}\Big) [\Phi^\pm_\lambda (x - \centerdot)]  \Big\|_{X_\lambda^\ast}, \]
where $\chi$ is a smooth cut-off such that $\chi (y) = 1$ for all $y \in \supp g$, the subindex $x$ in $\nabla_x$ indicates that the gradient is acting on the function $x \mapsto \Phi^\pm_\lambda (x - y)$. It remains to prove that
\begin{equation}
\sup_{|x|=R} \Big\| \chi \Big( \frac{x}{|x|} \cdot \nabla_x \mp i \lambda^{1/2}\Big) [\Phi^\pm_\lambda (x - \centerdot)]  \Big\|_{X_\lambda^\ast} = o(R^{-\frac{d-1}{2}}).
\label{id:finalCOND}
\end{equation}
To do so, the first point we should notice is that
\begin{equation}
\begin{aligned}
\Big\| \chi \Big( \frac{x}{|x|} \cdot \nabla_x \mp i \lambda^{1/2}\Big)& \, [\Phi^\pm_\lambda (x - \centerdot)] \Big\|_{X_\lambda^\ast} \\
& \lesssim \sum_{|\alpha| \leq 1} \lambda^{\frac{1 - |\alpha|}{2}} \Big\| \chi \Big( \frac{x}{|x|} \cdot \nabla_x \mp i \lambda^{1/2}\Big) [\partial^\alpha \Phi^\pm_\lambda (x - \centerdot)]  \Big\|_{L^2}
\end{aligned}
\label{in:XastL2H1}
\end{equation}
where $\alpha = (\alpha_1, \dots, \alpha_d) \in \N_0^d$ denotes a multi-index and $|\alpha| = \alpha_1 +\dots + \alpha_d$. This inequality follows from the inequality
\begin{equation}
\| u \|_{X^\ast_\lambda} \lesssim \lambda^{1/4} \vvvert u \vvvert_\ast + \| (\lambda - \Delta)^{1/2} u \|_{L^2}
\label{in:controllingXastLambda}
\end{equation}
where $(\lambda - \Delta)^{1/2}$ denotes the multiplier with symbol $(\lambda + |\xi|^2)^{1/2}$. The inequality \eqref{in:controllingXastLambda} is a consequence of a combination of three facts. The first one is the boundedness of $P_k$ with respect to the norm $\vvvert \centerdot \vvvert_\ast $. The second one is the inequality
$$ \lambda^{d/2 (1/q_d - 1/p_d)} \| P_k u \|_{L^{q_d}} \lesssim 2^k \| P_k u \|_{L^2} $$
for $k \in I$ ---which follows from Bernstein's inequality and the equivalence $2^k \simeq 2^{k_\lambda} \simeq \lambda^{1/2}$ when $k \in I$. The third fact is that $$ m_\lambda(\xi)^{1/2} |\widehat{P_{<I} u}(\xi)| \simeq \lambda^{1/2} |\widehat{P_{<I} u}(\xi)|,$$
and
$$ m_\lambda(\xi)^{1/2} |\widehat{P_k u}(\xi)| \simeq 2^k |\widehat{P_k u}(\xi)|$$
if $k>k_\lambda+1$. Combining these three facts, one can derive the inequality \eqref{in:controllingXastLambda}. Finally, the condition \eqref{id:finalCOND} follows from the inequality \eqref{in:XastL2H1} and the identity
\begin{equation}
\sup_{|x|=R} \Big| \frac{x}{|x|} \cdot \nabla_x \big(\partial^\alpha_y \Phi^\pm_\lambda (x - y) \big) \mp i \lambda^{1/2} \big( \partial^\alpha_y \Phi^\pm_\lambda (x - y) \big) \Big| = o_y(R^{-\frac{d-1}{2}})
\label{id:SRC_fundamental-sol}
\end{equation}
which holds uniformly for $y$ in compact subsets. The identity \eqref{id:SRC_fundamental-sol} for $\alpha = 0$ is the standard radiation condition. The case $|\alpha| = 1$ is known but might not be so standard. It is consequence of a tedious computation, that is actually, the exactly same computation used to show that
\[\sup_{|x|=R} \Big| \frac{x}{|x|} \cdot \nabla_x \big( \nu_y \cdot \nabla_y \Phi^\pm_\lambda (x - y) \big) \mp i \lambda^{1/2} \big( \nu_y \cdot \nabla_y \Phi^\pm_\lambda (x - y) \big) \Big| = o_y(R^{-\frac{d-1}{2}}),\]
where $\nu$ denotes the unitary exterior vector normal to the boundary of a smooth bounded domain. The last identity is rather standard and is the basic ingredient to show that, if a solution of the homogeneous equation $(\Delta + \lambda) u = 0$ in a exterior smooth bounded domain $\Omega = \R^d \setminus \overline{D}$ satisfies an integral representation in $\Omega$, in terms of its values and those of $\Phi^\pm_\lambda (x - \centerdot)$ on $\partial \Omega$, then $u$ has to satisfies the corresponding SRC. This shows that \eqref{id:finalCOND} holds and consequently the proof of this proposition is over.
\end{proof}

The next step will be to construct the scattering solution $u_{\rm sc}^\pm (\centerdot, y)$ as solution of the equation
\begin{equation}
\big(I - (\Delta + \lambda \pm i 0 - V^0)^{-1} \circ (\alpha\, \dd \sigma)\big)u_{\rm sc}^\pm(\centerdot, y) = f^\pm(\centerdot, y) \, \, \textnormal{in}\, \R^d
\label{eq:Fredholm}
\end{equation}
with $f^\pm(\centerdot, y) = (\Delta + \lambda \pm i 0 - V^0)^{-1}( V u_{\rm in}^\pm(\centerdot, y))$. Note that testing the operator $(\Delta + \lambda - V^0)$ with both sides of the identity \eqref{eq:Fredholm}, and applying the proposition \ref{prop:resolventV0}, we have that $u_{\rm sc}^\pm(\centerdot, y)$ solves the equation
\[ (\Delta + \lambda - V)u_{\rm sc}^\pm(\centerdot, y) = V u_{\rm in}^\pm(\centerdot, y) \, \, \textnormal{in}\, \R^d. \]
Moreover, since
\begin{equation}
u_{\rm sc}^\pm(\centerdot, y) = (\Delta + \lambda \pm i 0 - V^0)^{-1} [ (\alpha\, \dd \sigma) u_{\rm sc}^\pm(\centerdot, y) + V u_{\rm in}^\pm(\centerdot, y)]
\label{id:usc}
\end{equation}
we also have, by the proposition \ref{prop:resolventV0} that $u_{\rm sc}^\pm(\centerdot, y)$ satisfies the Sommerfeld radiation condition:
\begin{equation}
\sup_{|x|=R} \Big| \frac{x}{|x|} \cdot \nabla u_{\rm sc}^\pm (x, y) \mp i \lambda^{1/2} u_{\rm sc}^\pm (x, y)  \Big| = o_y(R^{-\frac{d-1}{2}}).
\label{cond:Sommerfeld}
\end{equation}
Thus, in order to prove the theorem \ref{th:scattering} is enough to solve the equation \eqref{eq:Fredholm}.

To invert the operator $\big(I - (\Delta + \lambda \pm i 0 - V^0)^{-1} \circ (\alpha\, \dd \sigma)\big)$ in $X_\lambda^\ast$ we use the Fredholm alternative. The first point
to be checked is the injectivity in $X_\lambda^\ast$ of the operator
\begin{equation}
\big(I - (\Delta + \lambda \pm i 0 - V^0)^{-1} \circ (\alpha\, \dd \sigma)\big).
\label{term:Fredholm_operator}
\end{equation}
The second point is to verify that $(\Delta + \lambda \pm i 0 - V^0)^{-1} \circ (\alpha\, \dd \sigma)$ is compact in $X_\lambda^\ast$.
%

Start by proving the compactness. By the proposition \ref{prop:resolventV0} it is sufficient to show that the multiplication by $\alpha\, \dd \sigma$ is compact from $X_\lambda^\ast$ to $X_\lambda$. Note that multiplication by $\alpha\, \dd \sigma$ is defined by
\[\langle f (\alpha \,\dd \sigma) , g \rangle = \langle \alpha \,\dd \sigma , f g \rangle = \int_\Gamma \alpha f g \, \dd \sigma. \]
Considering $\chi \in \mathcal{S}(\R^d)$ so that it does not vanish on $\Gamma$, we can write
\[\langle f (\alpha \,\dd \sigma) , g \rangle = \int_\Gamma \frac{\alpha}{\chi} (\chi f) g \, \dd \sigma, \]
which means that the operator multiplication by $\alpha \,\dd \sigma$ 
can be factorized as a composition of three operators, multiplication by $\chi$, restriction to $\Gamma$ ---trace operator--- and multiplication by $\alpha/\chi \,\dd \sigma $. Multiplication by $\alpha/\chi \,\dd \sigma$ is bounded from $L^2(\Gamma)$ to $X_\lambda$. This is a straightforward consequence of the Cauchy--Schwarz inequality, the theorem \ref{th:trace} and the definition \ref{def:Ylambdaast} ---in the sections \ref{sec:trace} and \ref{sec:refined}, respectively--- and the proposition
\ref{prop:Xlambdaast}. On the other hand, the trace on $\Gamma$ is a bounded 
operator from $\dot{B}^{1/2}_{2, 1}(\R^d)$ to $L^2(\Gamma)$ ---this is a 
Besov-space form of Theorem 14.1.1 in \cite{zbMATH02123716}. Recall that the
semi-norm of the homogeneous Besov space
$\dot{B}^{1/2}_{2, 1}(\R^d)$ is given by
\[ \| f \|_{\dot{B}^{1/2}_{2, 1}} = \sum_{l\in\Z} 2^{l/2}\|P_l f\|_{L^2}. \]
Finally, 
multiplication by $\chi$ is a compact operator 
from $X_\lambda^\ast$ to $\dot{B}^{1/2}_{2, 1}(\R^d)$
at least when $\chi$ is defined by
\begin{equation*}
\chi(x) = \frac{1}{(2 \pi)^{d /2}} \int_{\R^d} e^{i \delta/R_0 x \cdot \xi} \phi(\xi) \, \dd \xi
\end{equation*}
with $\phi \in \mathcal{S}(\R^d)$ be a $[0,1]$-valued function 
supported in $ \{ \xi \in \R^d: |\xi| \leq 1 \}$ and it is not 
identically zero, and $\delta \in (0, 1]$ chosen so that
\[ \Big| \int_{\R^d} e^{i x \cdot \xi} \phi(\xi) \, \dd \xi \Big| \geq  \frac{1}{2} \int_{\R^d} \phi(\xi) \, \dd \xi > 0 \]
whenever $|x| \leq  \delta$. 
The compactness is a consequence 
of the lemma \ref{lem:compact} and the definition \ref{def:Ylambdaast} ---in the sections \ref{sec:trace} and \ref{sec:refined}, respectively--- and the 
proposition \ref{prop:Xlambdaast}. Therefore, the operator multiplication by 
$\alpha \, \dd \sigma$ is a compact operator from $X_\lambda^\ast$ to 
$X_\lambda$. This conclude the proof of the compactness of
$(\Delta + \lambda \pm i 0 - V^0)^{-1} \circ (\alpha\, \dd \sigma)$ in 
$X_\lambda^\ast$.

Continue by proving the injectivity. Let $v^\pm (\centerdot, y) \in X_\lambda^\ast$ be in the kernel of \eqref{term:Fredholm_operator} and note that it satisfies that
\begin{equation}
v^\pm(\centerdot, y) = (\Delta + \lambda \pm i 0 - V^0)^{-1} [ (\alpha\, \dd \sigma) v^\pm(\centerdot, y)].
\label{id:vpm}
\end{equation}
Hence, by the proposition \ref{prop:resolventV0}, $v^\pm(\centerdot, y)$ satisfies the Sommerfeld radiation condition \eqref{cond:Sommerfeld}. Furthermore, testing $(\Delta + \lambda - V^0)$ with $v^\pm(\centerdot, y)$, and using the identities \eqref{id:vpm} and the proposition \ref{prop:resolventV0}, we obtain that $v^\pm(\centerdot, y)$ is solution of the equation
\[(\Delta + \lambda - V ) v^\pm(\centerdot, y) = 0  \, \, \textnormal{in}\, \R^d. \]
A direct application of the lemma \ref{lem:UniCont} below will show that $v^\pm(\centerdot, y)$ has to be identically zero. For that will need to show that $v^\pm$ belongs to $H^1_{\rm loc}(\R^d)$, which is a consequence of the inclusion $X^\ast_\lambda \subset H^1_{\rm loc}(\R^d)$ proved in the lemma \ref{lem:inclusion} below as well. Thus, we can use the Fredholm alternative to invert the operator \eqref{term:Fredholm_operator}, and construct $u_{\rm sc}^\pm(\centerdot, y)$ solving the equation \eqref{eq:Fredholm}. As we have already explained, this is the scattering solution we wanted, which ends the proof of the existence part of the theorem \ref{th:scattering}. The uniqueness part is again a direct application of the lemmas \ref{lem:UniCont} and \ref{lem:inclusion}.
%

\begin{lemma}\label{lem:UniCont}\sl Consider $d \geq 3$. If $u^\pm \in H^1_{\rm loc}(\R^d)$ is a solution of
\[(\Delta + \lambda - V ) u^\pm = 0  \, \, \textnormal{in}\, \R^d \]
that satisfies the radiation condition 
\[\sup_{|x|=R} \Big| \frac{x}{|x|} \cdot \nabla u^\pm (x) \mp i \lambda^{1/2} u^\pm (x)  \Big| = o(R^{-\frac{d-1}{2}}) \]
for all $R \geq R_0$, then $u^\pm$ has to be identically zero. 
\end{lemma}
\begin{proof}
The restriction of $u^\pm$ to $\R^d \setminus \supp V$ is solution of $(\Delta + \lambda) u^\pm = 0$ in $\R^d \setminus \supp V$. By Theorem 11.1.1 in \cite{zbMATH02123716} this restriction is smooth, and we have that
\begin{equation}
\int_{\partial B} \big| \partial_\nu u^\pm \mp i \lambda^{1/2} u^\pm \big|^2 \, \dd S = \int_{\partial B} | \partial_\nu u^\pm |^2 + \lambda | u^\pm |^2 \mp 2\lambda^{1/2} \Im(\partial_\nu u^\pm \overline{u^\pm})  \, \dd S
\label{id:SRC-Rellich}
\end{equation}
where $\partial_\nu = \nu \cdot \nabla$ with $\nu$ the exterior unit normal vector to $\partial B $ ---the boundary of $ B = \{ x\in \R^d: |x| < R \}$--- and $\Im$ denotes the imaginary part.
Extending $\nu$ to be the exterior unit normal vector to $\partial (B \setminus \overline{B_0})$ and integrating by parts in $B \setminus \overline{B_0}$, we have that
\begin{align*}
2\int_{\partial B} \Im(\partial_\nu u^\pm \overline{u^\pm})  \, \dd S &= -i \int_{\partial B} \partial_\nu u^\pm \overline{u^\pm} - \overline{\partial_\nu u^\pm} u^\pm \, \dd S \\
&= i \int_{\partial B_0} \partial_\nu u^\pm \overline{u^\pm} - \overline{\partial_\nu u^\pm} u^\pm \, \dd S = -2\int_{\partial B_0} \Im(\partial_\nu u^\pm \overline{u^\pm})  \, \dd S.
\end{align*}
Thus, taking limit, when $R$ goes to infinity, in the identity \eqref{id:SRC-Rellich} yields
\[\lim_{R\to\infty} \int_{\partial B} | \partial_\nu u^\pm |^2 + \lambda | u^\pm |^2 \, \dd S = \mp 2\lambda^{1/2} \int_{\partial B_0} \Im(\partial_\nu u^\pm \overline{u^\pm})  \, \dd S, \]
by the corresponding SRC. Since we assumed $V^0$ and $\alpha$ to be real-valued, we have integrating by parts now in $B_0$ that $\int_{\partial B_0} \Im(\partial_\nu u^\pm \overline{u^\pm})  \, \dd S = 0$, which implies that $\lim_{R\to\infty} \int_{\partial B} | u^\pm |^2 \, \dd S = 0$, and consequently, by Rellich's lemma, $\supp u^\pm \subset \overline{B_0}$ and $u^\pm \in H^1(\R^d)$.

It remains to prove that $u^\pm$ also vanishes in $B_0$, we do it using a Carleman estimate that Caro and Rogers proved in \cite{zbMATH06534426}. This estimate holds for a modified family of Bourgain-type spaces whose norms were
\[\| u \|_{Y^s_{\tau, M}} = \| (M \tau^2 + M^{-1} |q_\tau|^2)^{s/2} \widehat{u} \|_{L^2} \]
with $M, \tau \in [1, \infty)$, $s \in \R$ and $q_\tau (\xi) = - |\xi|^2 + i 2\tau \xi_d + \tau^2$. The estimate, stated in Theorem 2.1 from \cite{zbMATH06534426}, reads as follows. Set $\varphi(x) = \tau x_d + M x_d^2/2$ and $R \geq 1$. There exists an absolute constant $C > 0$, such that, if $M > CR^2$, then
\begin{equation}
\| u \|_{Y_{\tau, M}^{1/2}} \leq CR \| e^\varphi \Delta (e^{-\varphi} u) \|_{Y_{\tau, M}^{-1/2}}
\label{in:CR}
\end{equation}
for all $u \in \mathcal{S}(\R^d)$ with $\supp u \subset \{ x \in \R^d: |x|< R \}$ and $\tau > 8MR$. This inequality can be perturbed to consider the operator $\Delta + \lambda - V$ tested in any function in $H^1(\R^d)$ with support in $\overline{B_0}$. Indeed, start by estimating $(\lambda - V) u$ in $Y^{-1/2}_{\tau, M}$, by duality, with $u \in \mathcal{S}(\R^d)$ supported in $B_0$:
\begin{equation}
\langle (\lambda - V)u, v \rangle = \lambda \int_{\R^d} uv \, + \int_{\R^d} V^0 uv \, + \int_\Gamma \alpha uv \, \dd\sigma.
\label{term:2Bestimated}
\end{equation}
The first term on the right-hand side of \eqref{term:2Bestimated} can be easily bounded by the Cauchy--Schwarz inequality
\[\big| \lambda \int_{\R^d} uv \, \Big| \leq \lambda \| u \|_{L^2} \| v \|_{L^2} \leq \lambda M^{-1/2} \tau^{-1} \| u \|_{Y^{1/2}_{\tau, M}} \| v \|_{Y^{1/2}_{\tau, M}}.\]
To estimate the second term on the right-hand side of \eqref{term:2Bestimated}, we do as in the corollary \ref{cor:multiplication}
$$
\int_{\R^d} V^0 uv \, = \int_{\R^d}\mathbf{1}_E V^0 uv \, + \int_{\R^d}\mathbf{1}_F V^0 uv \,,
$$
where $E=\{x\in\R^d : |V^0(x)|\leq N\}$, $F=\R^d\setminus E$ and $N$ to be chosen. Thus, we have by the Cauchy--Schwarz and H\"older's inequalities 
\begin{align*}
\Big| \int_{\R^d} V^0 uv \, \Big| &\leq N \| u \|_{L^2} \| v \|_{L^2} + \| \mathbf{1}_F V^0 \|_{L^{d/2}} \| u \|_{L^{p_d}} \| v \|_{L^{p_d}}\, \\
&\leq N M^{-1/2} \tau^{-1} \| u \|_{Y^{1/2}_{\tau, M}} \| v \|_{Y^{1/2}_{\tau, M}} + \| \mathbf{1}_F V^0 \|_{L^{d/2}} \| u \|_{\dot{Y}^{1/2}_\tau} \| v \|_{\dot{Y}^{1/2}_\tau}.
\end{align*}
In the last inequality we have used Haberman's embedding ---see the corollary \ref{cor:H} in the section \ref{sec:others}. By the definition of the norm of the space $Y^{1/2}_{\tau, M}$ we have that
\[ \Big| \int_{\R^d} V^0 uv \, \Big| \leq ( N M^{-1/2} \tau^{-1} + M^{1/2} \| \mathbf{1}_F V^0 \|_{L^{d/2}} ) \| u \|_{Y^{1/2}_{\tau, M}} \| v \|_{Y^{1/2}_{\tau, M}}. \]
Finally, we estimate the third term on the right-hand side of the identity \eqref{term:2Bestimated}. To do so, we use the Besov-space form of Theorem 14.1.1 in \cite{zbMATH02123716} and see that
\begin{align*}
\Big| \int_\Gamma \alpha uv \, \dd\sigma \Big| \leq \| \alpha\|_{L^\infty(\Gamma)} \Big(&  \sum_{k \leq l_\tau + 1} 2^{k/2} \| P_k u \|_{L^2} \sum_{l \leq l_\tau + 1} 2^{k/2} \| P_l v \|_{L^2} \\
& + \sum_{k \leq l_\tau + 1} 2^{k/2} \| P_k u \|_{L^2} \sum_{l > l_\tau + 1} 2^{l/2} \| P_l v \|_{L^2} \\
& + \sum_{k > l_\tau + 1} 2^{k/2} \| P_k u \|_{L^2} \sum_{l \in \Z} 2^{l/2} \| P_l v \|_{L^2} \Big),
\end{align*}
where $l_\tau \in \Z$ satisfies that $2^{l_\tau - 1} < \tau \leq 2^{l_\tau}$. If $k>l_\tau + 1$, we have that $ 2^{k/2} |\widehat{P_k u}(\xi)| \simeq 2^{-k/2} |q_\tau(\xi)|^{1/2} |\widehat{P_k u}(\xi)|$ for all $\xi \in \R^d$. Hence, for the high frequencies we have
\begin{align*}
\sum_{k > l_\tau + 1} 2^{k/2} \| P_k u\|_{L^2} & \simeq \sum_{k > l_\tau + 1} 2^{-k/2}\||q_\tau|^{1/2} \widehat{P_k u}\|_{L^2} \\
& \lesssim \tau^{-1/2} \| u \|_{\dot{Y}^{1/2}_\tau} \leq \tau^{-1/2} M^{1/4} \| u \|_{Y^{1/2}_{\tau, M}}.
\end{align*}
On the other hand, for the low frequencies we have that
\[\sum_{k \leq l_\tau + 1} 2^{k/2} \| P_k u \|_{L^2} \lesssim \tau^{1/2} \|  u\|_{L^2} \leq M^{-1/4} \| u \|_{Y^{1/2}_{\tau, M}}. \]
Combining the previous inequalities for the high and low frequencies we obtain that there exists an absolute constant $C^\prime > 0$ such that
\[ \Big| \int_\Gamma \alpha uv \, \dd\sigma \Big| \leq C^\prime \| \alpha\|_{L^\infty(\Gamma)} ( M^{-1/2} + \tau^{-1/2} +\tau^{-1} M^{1/2} ) \| u \|_{Y^{1/2}_{\tau, M}} \| v \|_{Y^{1/2}_{\tau, M}}.  \]
We now choose $M$ so that $CR_0C^\prime \|\alpha \|_{L^\infty(\Gamma)} M^{-1/2} \leq 1/4$, then we choose $N$ such that $C R_0 M^{1/2} \| \mathbf{1}_F V^0 \|_{L^{d/2}} \leq 1/4$, and finally, we consider $\tau$ to have
$$C R_0 [(\lambda + N) M^{-1/2} \tau^{-1} + C^\prime \| \alpha\|_{L^\infty(\Gamma)} ( \tau^{-1/2} +\tau^{-1} M^{1/2} ) ] < 1/4. $$
Therefore, we can conclude that there exists a $\tau_0 = \tau_0 (R_0, \| \alpha\|_{L^\infty(\Gamma)}, V^0, \lambda)$  such that
\begin{equation}
\| u \|_{Y_{\tau, M}^{1/2}} \leq 4CR_0 \| e^\varphi \big( \Delta + \lambda - V \big) (e^{-\varphi} u) \|_{Y_{\tau, M}^{-1/2}}
\label{es:a-priori}
\end{equation}
for all $u \in \mathcal{S}(\R^d)$ with $\supp u \subset B_0$ and $\tau > \tau_0$. One can check that $Y^{1/2}_{\tau, M}$ and $H^1(\R^d)$ are equal as sets, and that, for every $u \in H^1(\R^d)$ with $\supp u \subset \overline{B_0}$, we have $e^\varphi \big( \Delta + \lambda - V \big) (e^{-\varphi} u) \in Y^{-1/2}_{\tau, M}$. Thus , by a density argument
\begin{equation}
\| e^\varphi u \|_{Y_{\tau, M}^{1/2}} \leq 4CR_0 \| e^\varphi \big( \Delta + \lambda - V \big) u \|_{Y_{\tau, M}^{-1/2}}
\label{in:Carleman}
\end{equation}
for all $u \in H^1(\R^d)$ with $\supp u \subset \overline{B_0}$ and $\tau > \tau_0$. Since $u^\pm$ is supported in $\overline{B_0}$, belongs to $H^1(\R^d)$ and solves $(\Delta + \lambda - V) u^\pm = 0$ in $\R^d$, we have that $u^\pm$ is identically zero by applying the inequality \eqref{in:Carleman}.
\end{proof}

\begin{lemma}\label{lem:inclusion}\sl Every $u \in X^\ast_\lambda $ belongs to $ H^1_{\rm loc}(\R^d)$.
\end{lemma}
\begin{proof}
Consider $u \in X^\ast_\lambda$ and set $u_I = \sum_{k \in I} P_k u$ and $u_{\Z \setminus I} = u - u_I$. Let us show that $u_I$ belongs to $ H^1_{\rm loc} (\R^d)$ and $u_{\Z \setminus I} $ is in $ H^1(\R^d)$. Let $K$ be a compact subset of
$\R^d$ and $\alpha = (\alpha_1, \dots, \alpha_d) \in \N_0^d$ denote a multi-index such that $|\alpha| = \alpha_1 +\dots + \alpha_d \leq 1$, we have that
\[\| \partial^\alpha u_I \|_{L^2(K)} \leq \sum_{k \in I} \| \partial^\alpha P_k u \|_{L^2(K)}. \]
By H\"older's inequality, Bernstein's inequality for $\partial^\alpha$, and the fact that $I$ only contains four elements, we have
\[ \| \partial^\alpha u_I \|_{L^2(K)} \lesssim \sum_{k \in I} \| \partial^\alpha P_k u \|_{L^{q_d}} 
 \lesssim \sum_{k \in I} 2^{|\alpha|k} \| P_k u \|_{L^{q_d}} \lesssim \Big( \sum_{k \in I} 2^{2k} \| P_k u \|_{L^{q_d}}^2 \Big)^\frac{1}{2}
.\]
Thus,
\begin{equation*}
\| \partial^\alpha u_I \|_{L^2(K)} \lesssim \lambda^{\frac{1}{2} - \frac{d}{2}(\frac{1}{q_d} - \frac{1}{p_d})} \Big( \sum_{k \in I} \lambda^{d(\frac{1}{q_d} - \frac{1}{p_d})} \| P_k u \|_{L^{q_d}}^2 \Big)^\frac{1}{2},
\end{equation*}
which shows that $u_I$ belongs to $H^1_{\rm loc} (\R^d)$. Next we prove that $u_{\Z \setminus I}$ belongs to $H^1(\R^d)$. Let $(\lambda - \Delta)^{1/2}$ denote the multiplier with symbol $(\lambda + |\xi|^2)^{1/2}$. By Plancherel's identity and the finite overlapping of the supports of $\{ P_k u : k > k_\lambda + 1 \}$, we have 
\begin{align*}
\| (\lambda - \Delta)^{1/2} u_{\Z \setminus I} & \|_{L^2}^2 =  \int_{\R^d} (\lambda + |\xi|^2) |\widehat {u_{\Z \setminus I}}(\xi)|^2 \, \dd \xi \\
& \simeq \int_{\R^d} (\lambda + |\xi|^2) |\widehat {P_{<I} u}(\xi)|^2 \, \dd \xi + \sum_{k > k_\lambda + 1} \int_{\R^d} (\lambda + |\xi|^2) |\widehat {P_k u}(\xi)|^2 \, \dd \xi \\
& \simeq \lambda \int_{\R^d} |\widehat {P_{<I} u}(\xi)|^2 \, \dd \xi + \sum_{k > k_\lambda + 1 } 2^{2k} \int_{\R^d} |\widehat {P_k u}(\xi)|^2 \, \dd \xi.
\end{align*}
Note that 
$ \lambda^{1/2} |\widehat{P_{<I} u}(\xi)| \simeq m_\lambda(\xi)^{1/2} |\widehat{P_{<I} u}(\xi)|$ for all $\xi \in \R^d$. While, if $k>k_\lambda+1$, we have that $ 2^k |\widehat{P_k u}(\xi)| \simeq m_\lambda(\xi)^{1/2} |\widehat{P_k u}(\xi)|$. Hence,
\begin{equation*}
\| (\lambda - \Delta)^{1/2} u_{\Z \setminus I} \|_{L^2}^2 \simeq \| m_\lambda^{1/2} \widehat{P_{<I} u} \|_{L^2}^2 + \sum_{k > k_\lambda+1} \| m_\lambda^{1/2} \widehat{P_k u} \|_{L^2}^2,
\end{equation*}
which proves that $u_{\Z \setminus I}$ belongs to $H^1(\R^d)$. This ends the proof of this lemma.
\end{proof}

We finish this section by stating an inequality which will be essential to address the inverse scattering problem.

\begin{lemma}\label{lem:a-priori}\sl Consider $d\geq 3$. There exist $M = M(R_0, \| \alpha \|_{L^\infty(\Gamma)})$, and positive constants $C = C(R_0)$ and $\tau_0 = \tau_0 (R_0, \| \alpha \|_{L^\infty(\Gamma)}, V^0, \lambda)$ such that
\begin{equation*}
\| u \|_{Y_{\tau, M}^{1/2}} \leq C \| e^\varphi \big( \Delta + \lambda - T^\ast V \big) (e^{-\varphi} u) \|_{Y_{\tau, M}^{-1/2}}
\end{equation*}
for all $u \in \mathcal{S}(\R^d)$ with $\supp u \subset B_0$ and $\tau > \tau_0$. Here $T^\ast V$ denotes the following potential
\[\langle T^\ast V, \phi \rangle = \int_{\R^d} T^\ast V^0 \, \phi \, + \int_{T^t \Gamma} T^\ast \alpha \, \phi \, \dd T^\ast\sigma,\]
where $T \in SO(d)$, $T^\ast V^0 (x) = V^0 (Tx)$, $T^\ast \alpha (x) = \alpha (Tx)$, $T^t \Gamma = \{ T^t x : x \in \Gamma \}$ and $T^\ast \sigma (E) = \sigma (TE)$ with $TE = \{ Tx : x \in E \}$.
\end{lemma}
\begin{proof}
We start from the inequality \eqref{in:CR} ---due to Caro and Rogers \cite{zbMATH06534426}--- and perturb it to include the potential $T^\ast V$. This procedure is exactly the same as the one used in proof of the lemma \ref{lem:UniCont} to derive the inequality \eqref{es:a-priori} and we will not repeat it.
\end{proof}

\section{Inverse scattering}\label{sec:inverse_scattering}
In this section we adapt to our framework the approach we learnt from 
\cite{zbMATH01670370} by H\"ahner and Hohage. The first step is to
obtain the orthogonality identity \eqref{id:integral}. In order to prove it, we need two lemmas regarding the single layer potential $\mathcal{S}^\pm$ whose kernel is given by the total wave
\[u^\pm_{\rm to}|_{\partial B_0 \times \partial B_0} = u^\pm_{\rm in} |_{\partial B_0 \times \partial B_0} + u^\pm_{\rm sc}|_{\partial B_0 \times \partial B_0}.\]
For $f$ continuous on $\partial B_0$, we define the \textit{single layer potential} as
\[\mathcal{S}^\pm f (x) = \int_{\partial B_0} u^\pm_{\rm to}(x,\centerdot) f  \, \dd S,\]
for $x\in \partial B_0$
where $\dd S$ denotes the volume form on $\partial B_0$.
\begin{lemma}\label{lem:single_layer}\sl The scattering solution of the theorem \ref{th:scattering} satisfies the following reciprocity relation
\[u^\pm_{\rm sc}(x,y) = u^\pm_{\rm sc}(y,x), \quad \forall x,y \in \R^d \setminus \supp V.\]
In particular, the single layer potential $\mathcal{S}^\pm$ is symmetric, that is,
\[\int_{\partial B_0} \mathcal{S}^\pm f g \, \dd S = \int_{\partial B_0} f \mathcal{S}^\pm g \, \dd S \]
for all $f$ and $g$ continuous on $\partial B_0$.
\end{lemma}

\begin{proof}
Given $x,y \in \R^d \setminus \supp V$, there exists a bounded domain $D$ containing $\supp V$ so that
$x,y \in \R^d \setminus \overline{D}$ and its boundary is locally 
described by the graphs of twice continuously differentiable 
functions. The restrictions of $u^\pm_{\rm sc}(\centerdot,x)$ and $u^\pm_{\rm sc}(\centerdot,y)$ to $\R^d \setminus \supp V$ are solutions of the equation $(\Delta + \lambda) u = 0$ in $\R^d \setminus \supp V$. By Theorem 11.1.1 in \cite{zbMATH02123716} these restrictions are smooth. Thus, integrating by parts in a $B \setminus \overline{D}$, with $B = \{ x \in \R^d : |x| < R \}$, and making $R$ go to infinity we have that
\begin{equation}
\int_{\partial D} \big[ \partial_\nu u^\pm_{\rm sc}(\centerdot,y) u^\pm_{\rm sc}(\centerdot,x) - u^\pm_{\rm sc}(\centerdot,y) \partial_\nu u^\pm_{\rm sc}(\centerdot,x) \big] \, \dd S = 0
\label{id:boundarySCSC}
\end{equation}
where $\dd S$ denotes the volume form on $\partial D$ and $\nu$ stands for the unit exterior normal vector on $\partial D$. In order to make the integration on $\partial B$ vanish when $R$ goes to infinity, we just need to use the corresponding SRC. On the other hand, since the restrictions of $u^\pm_{\rm in}(\centerdot,x)$ and $u^\pm_{\rm in}(\centerdot,y)$ to $D$ are solutions of the equation $(\Delta + \lambda) u = 0$ in $D$, we have, integrating by parts in $D$, that
\begin{equation}
\int_{\partial D} \big[ \partial_\nu u^\pm_{\rm in}(\centerdot,y) u^\pm_{\rm in}(\centerdot,x) - u^\pm_{\rm in}(\centerdot,y) \partial_\nu u^\pm_{\rm in}(\centerdot,x) \big] \, \dd S = 0.
\label{id:boundaryININ}
\end{equation}
Finally, it is well-known that smooth solutions of $(\Delta + \lambda) u = 0$ in $\R^d \setminus \supp V$ can be represented by a boundary 
integral expression. In particular, since $u^\pm_{\rm in}(\centerdot, z) = \Phi^\pm_\lambda (z - \centerdot)$ the functions $u^\pm_{\rm sc}(\centerdot,y)$ and $u^\pm_{\rm sc}(\centerdot,x)$ can be represented
respectively by
\begin{equation}
u^\pm_{\rm sc}(z,y) = \int_{\partial D} \big[ \partial_\nu u^\pm_{\rm sc}(\centerdot,y) u^\pm_{\rm in}(\centerdot,z) - u^\pm_{\rm sc}(\centerdot,y) \partial_\nu u^\pm_{\rm in}(\centerdot,z) \big] \, \dd S , \quad \forall z \in \R^d \setminus \overline{D},
\label{id:boundarySCIN}
\end{equation}
and
\begin{equation}
u^\pm_{\rm sc}(z,x) = \int_{\partial D} \big[ \partial_\nu u^\pm_{\rm sc}(\centerdot,x) u^\pm_{\rm in}(\centerdot,z) - u^\pm_{\rm sc}(\centerdot,x) \partial_\nu u^\pm_{\rm in}(\centerdot,z) \big] \, \dd S , \quad \forall z \in \R^d \setminus \overline{D}.
\label{id:boundaryINSC}
\end{equation}
Note that evaluating \eqref{id:boundarySCIN} at $z = x$ and \eqref{id:boundaryINSC} at $z = y$, we can compute $u^\pm_{\rm sc}(x,y) - u^\pm_{\rm sc}(y,x)$. Now, using the identities \eqref{id:boundarySCSC} and \eqref{id:boundaryININ} we have
\[u^\pm_{\rm sc}(x,y) - u^\pm_{\rm sc}(y,x) = \int_{\partial D} \big[ \partial_\nu u^\pm_{\rm to}(\centerdot,y) u^\pm_{\rm to}(\centerdot,x) - u^\pm_{\rm to}(\centerdot,y) \partial_\nu u^\pm_{\rm to}(\centerdot,x) \big] \, \dd S. \]
Integrating by parts the right-had side of last identity in $D$ and using that $u^\pm_{\rm to}(\centerdot,y) $ and $ u^\pm_{\rm to}(\centerdot,x)$ are solutions of $(\Delta + \lambda - V) u = 0$ in $D$, we get that
\[u^\pm_{\rm sc}(x,y) - u^\pm_{\rm sc}(y,x) = \langle V u^\pm_{\rm to}(\centerdot,y),  u^\pm_{\rm to}(\centerdot,x) \rangle - \langle V u^\pm_{\rm to}(\centerdot,x), u^\pm_{\rm to}(\centerdot,y) \rangle = 0. \]
This finishes the proof of the first part of this lemma. The second part is a direct consequence of first one since
\[u^\pm_{\rm to}|_{\partial B_0 \times \partial B_0} = u^\pm_{\rm in} |_{\partial B_0 \times \partial B_0} + u^\pm_{\rm sc}|_{\partial B_0 \times \partial B_0}\]
is the kernel of the single layer potential and $u_{\rm in}^\pm(x, y) = \Phi_\lambda^\pm (y - x) $ with $\Phi_\lambda^\pm$ radially symmetric.
\end{proof}

\begin{lemma}\label{lem:kernel_solution}\sl Consider $d \geq 3$. Let $f$ be continuous on $\partial B_0$. Then, the function
\[u^\pm (x) = \int_{\partial B_0} u^\pm_{\rm to}(x,\centerdot) f  \, \dd S \]
is the unique solution in $H^1_{\rm loc}(\R^d)$ of the problem
\begin{equation}
\left\{
\begin{aligned}
& (\Delta + \lambda - V) u^\pm = 0 & & \textsl{in}\, \R^d \setminus \partial B_0,\\
& \partial_\nu u^\pm|_{\R^d \setminus \overline{B_0}} - \partial_\nu u^\pm|_{B_0} = f & & \textsl{on}\, \partial B_0,\\
& \sup_{|x|=R} \Big| \frac{x}{|x|} \cdot \nabla u^\pm (x) \mp i \lambda^{1/2} u^\pm (x)  \Big| = o(R^{-\frac{d-1}{2}})  & & \textsl{for}\, R\geq R_0.
\end{aligned}
\right.
\label{pb:extension}
\end{equation}
Here $\nu$ is the unit exterior normal vector on $\partial B_0$.
\end{lemma}
\begin{proof}
Start by showing that the problem \eqref{pb:extension} has a unique solution in $H^1_{\rm loc}(\R^d)$.
Note that it is enough to show that if $u^\pm $ is a solution 
for $f = 0$, then $u^\pm = 0$.
The restrictions of $u^\pm$ to $B_0 \setminus \supp V$ and $\R^d \setminus \overline{B_0}$ are solutions of the equation $(\Delta + \lambda) u^\pm = 0$ in $B_0 \setminus \supp V$ and $\R^d \setminus \overline{B_0}$ respectively. By Theorem 11.1.1 in \cite{zbMATH02123716} these restrictions are smooth and can be extended by continuity up to the boundary of $B_0$. Since $u^\pm \in H^1_{\rm loc}(\R^d)$, the extensions from both sides of the boundary must coincide.
The facts that $f=0$ and $u^\pm$ is continuous across $\partial B_0$ make $u^\pm$ be a 
solution of $(\Delta + \lambda - V) u^\pm = 0$ in $\R^d$. Since it satisfies the SRC and belongs to $H^1_{\rm loc} (\R^d)$, it has to vanish everywhere by the lemma \ref{lem:UniCont}.

Now we show that $u^\pm$ is solution of \eqref{pb:extension}.
The function $u^\pm$ belongs to $H^1_{\rm loc}(\R^d)$ because $u^\pm_{\rm sc}(\centerdot,y)$ is in $X^\ast_\lambda$ (recall the lemma \ref{lem:inclusion}) and the function
\[v^\pm (x) = \int_{\partial B_0} u^\pm_{\rm in}(x,\centerdot) f  \, \dd S \]
is continuous in $\R^d$ and smooth in $\R^d \setminus \partial B_0$ ---recall that $u^\pm_{\rm in}(\centerdot, y = \Phi^\pm_\lambda (y - \centerdot))$. Moreover, $u^\pm$ solves the problem \eqref{pb:extension} because $u^\pm_{\rm sc}(\centerdot,y)$ is smooth in $\R^d \setminus \supp V$ for $y \in \R^d \setminus B_0$ (Theorem 11.1.1 in \cite{zbMATH02123716}) and solves the problem \eqref{pb:scattering}, and $v^\pm$ solves the problem
\begin{equation*}
\left\{
\begin{aligned}
& (\Delta + \lambda) v^\pm = 0 & & \textsl{in}\, \R^d \setminus \partial B_0,\\
& \partial_\nu v^\pm|_{\R^d \setminus \overline{B_0}} - \partial_\nu v^\pm|_{B_0} = f & & \textsl{on}\, \partial B_0,\\
& \sup_{|x|=R} \Big| \frac{x}{|x|} \cdot \nabla v^\pm (x) \mp i \lambda^{1/2} v^\pm (x)  \Big| = o(R^{-\frac{d-1}{2}})  & & \textsl{for}\, R\geq R_0,
\end{aligned}
\right.
\end{equation*}
---again the fact that $v^\pm$ solves this problem is classical (see for example \cite{zbMATH06061716}).
\end{proof}

\begin{proposition}\label{prop:int_identity} \sl Consider $d\geq 3$. Let $V_1$ and $V_2$ be two electric potentials as in the theorem \ref{th:uniqueness}. Let $ u^\pm_{{\rm sc},1} (\centerdot, y) $ and $ u^\pm_{{\rm sc},2}  (\centerdot, y) $ with $y \in \R^d \setminus B_0 $ be the scattering solutions of the theorem \ref{th:scattering} associated to $V_1$ and $V_2$. If
\[u^\pm_{{\rm sc},1}|_{\partial B_0 \times \partial B_0} = u^\pm_{{\rm sc},2}|_{\partial B_0 \times \partial B_0},\]
then
\[\langle (V_1 - V_2) v_1, v_2 \rangle = 0 \]
for all $v_1$ and $v_2$ in $H^1 (B_0)$ such that $(\Delta + \lambda - V_j) v_j = 0$ in $B_0$.
\end{proposition}
\begin{proof}
By Theorem 11.1.1 in \cite{zbMATH02123716}, we know that the 
restriction of $v_j$ to $B_0 \setminus \supp V_j$ is smooth. We 
extend $v_j$ up to $\partial B_0$ by continuity. Let $w_j^\pm$ be
the solution of $(\Delta + \lambda)w_j^\pm = 0$ in
$\R^d \setminus \overline{B_0}$ satisfying the corresponding SRC and 
the Dirichlet boundary condition
$w_j^\pm|_{\partial B_0} = v_j|_{\partial B_0}$. The solution $w^\pm_j$ is continuous in $\R^d \setminus B_0$ and smooth in $\R^d \setminus \overline{B_0}$ (see Theorem 3.11 in 
\cite{zbMATH06061716}).
Then, the function 
\[u_j^\pm = \mathbf{1}_{\overline{B_0}} v_j + \mathbf{1}_{\R^d\setminus \overline{B_0}} w_j^\pm \in L^1_{\rm loc} (\R^d) \]
---$\mathbf{1}_{\overline{B_0}}$ and $\mathbf{1}_{\R^d\setminus \overline{B_0}}$ stand for the characteristic functions of $\overline{B_0}$ and its complement--- and, by the lemma \ref{lem:kernel_solution}, satisfies that
\[u_j^\pm (x) = \int_{\partial B_0} u^\pm_{{\rm to},j}(x,\centerdot)
\big(\partial_\nu w_j^\pm - \partial_\nu v_j\big)  \, \dd S,\]
where $\nu$ is the unit exterior normal vector on $\partial B_0$. In particular,
\begin{equation}
v_j|_{\partial B_0} = \mathcal{S}_j^\pm \big(\partial_\nu w_j^\pm - \partial_\nu v_j\big).
\label{id:solution_singleL}
\end{equation}
Note that, integrating by parts in $B_0$ we have that
\[ \langle (V_1 - V_2) v_1, v_2 \rangle = \int_{\partial B_0} \big( v_2\partial_\nu v_1 - v_1 \partial_\nu v_2 \big) \dd S;\]
while integrating by parts in $B \setminus \overline{B_0}$, where $B = \{ x \in \R^d : |x| < R \}$, and making $R$ go to infinity we have that
\begin{equation}
\int_{\partial B_0} \big( v_2\partial_\nu w_1^\pm - v_1 \partial_\nu w_2^\pm \big) \dd S = 0
\label{id:trivial}
\end{equation}
by the SRC. Then, by the identity \eqref{id:trivial} first and then 
by \eqref{id:solution_singleL}, we 
have that
\begin{align*}
\langle (V_1 - V_2) v_1, v_2 \rangle = - \int_{\partial B_0} \big[ & v_2 \big(\partial_\nu w_1^\pm - \partial_\nu v_1\big) - v_1 \big(\partial_\nu w_2^\pm - \partial_\nu v_2\big) \big] \dd S \\
=  - \int_{\partial B_0} \big[& \mathcal{S}_2^\pm \big(\partial_\nu w_2^\pm - \partial_\nu v_2\big) \big(\partial_\nu w_1^\pm - \partial_\nu v_1\big) \\
& - \mathcal{S}_1^\pm \big(\partial_\nu w_1^\pm - \partial_\nu v_1\big) \big(\partial_\nu w_2^\pm - \partial_\nu v_2\big) \big] \dd S.
\end{align*}
By the symmetry of $\mathcal{S}_j^\pm$ stated in the lemma \ref{lem:single_layer}, we have
\[\langle (V_1 - V_2) v_1, v_2 \rangle = \int_{\partial B_0} \big[\mathcal{S}_1^\pm - \mathcal{S}_2^\pm\big] \big(\partial_\nu w_2^\pm - \partial_\nu v_2\big) \big(\partial_\nu w_1^\pm - \partial_\nu v_1\big) \dd S. \]
Thus, if $u^\pm_{{\rm sc},1}|_{\partial B_0 \times \partial B_0} = u^\pm_{{\rm sc},2}|_{\partial B_0 \times \partial B_0}$ the kernel of the operator $\mathcal{S}_1^\pm - \mathcal{S}_2^\pm$ is zero, and consequently, $\langle (V_1 - V_2) v_1, v_2 \rangle = 0$.
\end{proof}

As we mentioned in the introduction, we will test the identity of the proposition \ref{prop:int_identity} with a family of CGO solutions of the form
\begin{equation}
v_j(x) = e^{\zeta_j \cdot x} (1 + w_j(x)),
\label{id:CGO}
\end{equation}
where $\zeta_j \in \C^d$ so that $\zeta_j \cdot \zeta_j = - \lambda$ 
and $\zeta_1 + \zeta_2 = -i\kappa$ for an arbitrarily given
$\kappa \in \R^d$, and the correction term $w_j$ vanishing in a 
specific sense when $|\zeta_j|$ grows. In order to state the existence of this type of solutions, we will need to introduce some spaces in the spirit of Haberman and Tataru in \cite{zbMATH06145493}, and Caro and Rogers in \cite{zbMATH06534426}. First we introduce the non-homogeneous Bourgain space $X^s_\zeta$ with $s \in \R$, which consists of $u \in \mathcal{S}^\prime (\R^d)$ so that $\widehat{u} \in L^2_{\rm loc}(\R^d)$ and
\[ \| u \|_{X^s_\zeta}^2 = \int_{\R^d} (|\zeta| + |p_\zeta(\xi)|)^{2s} |\widehat{u}(\xi)|^2 \, \dd \xi < \infty, \]
endowed with the norm $\| \centerdot \|_{X^s_\zeta}$. Here $p_\zeta (\xi) = - |\xi|^2 + i 2 \zeta \cdot \xi$ for $\xi \in \R^d$. Then,
for $s\geq 0$, we introduce the space
\[X^s_\zeta (B_0) = \{ u|_{B_0} : u \in X^s_\zeta \}\]
endowed with the norm
\[\| u \|_{X^s_\zeta (B_0)} = \inf \{ \| v \|_{X^s_\zeta} : v|_{B_0} = u \}. \]
For us, the only relevant indices will be $s = 1/2$ and $s= -1/2$. In 
addition to these spaces, there is another family of spaces that will be 
useful for us. This is given, for $s \in \R$, by the set
\[X^s_{\zeta, c}(B_0) = \{ u \in X^s_\zeta : \supp u \subset \overline{B_0} \},\]
endowed with the same norm $\| \centerdot \|_{X^s_\zeta}$. As it was stated 
in  \cite{zbMATH06534426}, $X^{-s}_{\zeta, c}(B_0)$ can be identified with dual space of $X^s_\zeta (B_0)$ for $s\geq 0$.

\begin{proposition}\label{prop:CGO}\sl Consider $d \geq 3$ and $\tau_0$ as in the lemma \ref{lem:a-priori}. For every $\zeta = \Re \zeta + i \Im \zeta \in \C^d $ such that $|\Re \zeta| = \tau$, $|\Im \zeta| = (\tau^2 + \lambda)^{1/2}$ and $\Re \zeta \cdot \Im \zeta = 0$ with $\tau \geq \tau_0$, we have that there exists $w_\zeta \in X^{1/2}_\zeta(B_0)$ so that $v_\zeta = e^{\zeta \cdot x} (1 + w_\zeta)$ is solution of the equation $(\Delta + \lambda - V)v_\zeta = 0$ in $B_0$ and
\[\| w_\zeta \|_{X^{1/2}_\zeta (B_0)} \lesssim \| V \|_{X^{-1/2}_\zeta}.\]
\end{proposition}
\begin{proof}
The lemma \ref{lem:a-priori} is the analogue of Lemma 2.1 in \cite{zbMATH06534426}. Then, considering $T \in SO(d)$ so that $\Re \zeta = \tau T e_d$ and arguing as in Lemma 2.2, Lemma 2.3 and Proposition 2.4 in \cite{zbMATH06534426} we can derive the following inequality:
\[\| u \|_{X^{1/2}_\zeta} \lesssim \| (\Delta + 2 \zeta \cdot \nabla - V)u \|_{X^{-1/2}_\zeta}\]
for all $u \in \mathcal{S}(\R^d)$ with $\supp u \subset B_0$ and $\tau > \tau_0$. The implicit constant in the previous inequality depends on $R_0$ and $\| \alpha \|_{L^\infty(\Gamma)}$, while $\tau_0$ is as in the lemma \ref{lem:a-priori}.

This inequality is the analogue of the one stated in Proposition 2.4 from \cite{zbMATH06534426}, and it represents the key ingredient to perform the method of \textit{a priori estimates} which yields the existence of $w_\zeta$ and its corresponding bound by the norm of $V$. For the details, see the pages 11 and 12, and Proposition 2.5 in \cite{zbMATH06534426}. 
\end{proof}

The proposition \ref{prop:CGO} yields directly pairs of solutions as 
in \eqref{id:CGO}, however, we also need that these pairs satisfy
$\zeta_1 + \zeta_2 = -i\kappa$ for an arbitrarily given $\kappa \in \R^d$. 
Thus, let $\kappa \in \R^d$ be given and choose $\eta, \theta \in \R^d$ so 
that $|\eta| = |\theta| = 1$ and
$\eta \cdot \theta = \eta \cdot \kappa = \theta \cdot \kappa = 0$. Then, for 
$\tau $ such that $\tau^2 \geq |\kappa|^2/4 - \lambda$ we set
\begin{equation}
\begin{aligned}
\zeta_1 & = \tau \eta + i \Big[ -\frac{\kappa}{2} + \Big( \tau^2 + \lambda - \frac{|\kappa|^2}{4} \Big)^{1/2} \theta \Big], \\
\zeta_2 & = - \tau \eta + i \Big[ -\frac{\kappa}{2} - \Big( \tau^2 + \lambda - \frac{|\kappa|^2}{4} \Big)^{1/2} \theta \Big].
\end{aligned}
\label{id:zeta}
\end{equation}
Since $\zeta_1$ and $\zeta_2$ satisfy the conditions of the proposition 
\ref{prop:CGO}, there exists solutions $v_1$ and $v_2$ as in 
\eqref{id:CGO} solving the equation $(\Delta + \lambda - V_j)v_j = 0$ 
in $B_0$ with $w_j \in X^{1/2}_{\zeta_j}(B_0)$ such that
\begin{equation}
\| w_j \|_{X^{1/2}_{\zeta_j} (B_0)} \lesssim \| V_j \|_{X^{-1/2}_{\zeta_j}}
\label{in:reminder}
\end{equation}
for $j = 1,2$ and $\tau^2 \geq \max\{\tau_0^2, |\kappa|^2/4 - \lambda\}$.
Considering any extension of 
$w_j$ in $X^{1/2}_{\zeta_j}(B_0)$ to $X^{1/2}_{\zeta_j}$, one can check that this extension is in $H^1(\R^d)$ and consequently, $w_j$ belongs to
$H^1(B_0)$ and so does $v_j$. Therefore, the solutions $v_1$ and $v_2$ can be 
plugged in to the identity of the proposition \ref{prop:int_identity}, and 
obtain that
\[\langle V_1 - V_2, e^{-i\kappa\cdot x} \rangle = - \langle V_1 - V_2 , e^{-i\kappa\cdot x} w_1 \rangle - \langle V_1 - V_2 , e^{-i\kappa\cdot x} w_2 \rangle - \langle (V_1 - V_2) w_1 , e^{-i\kappa\cdot x} w_2 \rangle.\]
The first two terms on the right-hand side can be bounded as follows
\[| \langle V_1 - V_2 , e^{-i\kappa\cdot x} w_j \rangle | \leq \| V_1 - V_2 \|_{X^{-1/2}_{\zeta_j}} \| e^{-i\kappa\cdot x} w_j \|_{X^{1/2}_{\zeta_j}(B_0)},\]
since $\supp (V_1 - V_2) \subset B_0$ and because of the duality between 
$X^{-1/2}_{\zeta_j, c} (B_0)$ and $X^{1/2}_{\zeta_j}(B_0)$. One can check 
that
\begin{equation}
\| e^{-i\kappa\cdot x} w_j \|_{X^{1/2}_{\zeta_j}(B_0)} \lesssim (1 + |\kappa|) \| w_j \|_{X^{1/2}_{\zeta_j}(B_0)}
\label{in:modulation}
\end{equation}
and consequently, by \eqref{in:modulation} and \eqref{in:reminder}, one obtains
\begin{equation}
| \langle V_1 - V_2 , e^{-i\kappa\cdot x} w_j \rangle | \lesssim (1 + |\kappa|) \| V_1 - V_2 \|_{X^{-1/2}_{\zeta_j}} \| V_j \|_{X^{-1/2}_{\zeta_j}}.
\label{in:1st2nd}
\end{equation}
On the other hand, the third term can be bounded, by duality, as follows
\[| \langle (V_1 - V_2) w_1 , e^{-i\kappa\cdot x} w_2 \rangle | \leq \| (V_1 - V_2) w_1 \|_{X^{-1/2}_{\zeta_2}} \| e^{-i\kappa\cdot x} w_2 \|_{X^{1/2}_{\zeta_2}(B_0)} \]
since again $(V_1 - V_2) w_1$ is supported in $B_0$. We will show that the operator multiplication by $V_1 - V_2$ is a bounded from $X^{1/2}_{\zeta_1}(B_0)$ to $X^{-1/2}_{\zeta_2}$. For the time being, let us assume that such boundedness holds. Then, we have by \eqref{in:modulation} and \eqref{in:reminder} that
\begin{equation}
\begin{aligned}
| \langle (V_1 - V_2) w_1 , e^{-i\kappa\cdot x} w_2 \rangle |
& \lesssim (1 + |\kappa|) \|w_1 \|_{X^{1/2}_{\zeta_1}(B_0)} \|w_2 \|_{X^{1/2}_{\zeta_2}(B_0)} \\
& \lesssim (1 + |\kappa|) \|V_1 \|_{X^{-1/2}_{\zeta_1}} \|V_2 \|_{X^{-1/2}_{\zeta_2}}.
\end{aligned}
\label{in:3rd}
\end{equation}
Gathering the inequalities \eqref{in:1st2nd} and \eqref{in:3rd}, one obtains the following bound
\begin{equation}
\big| \langle V_1 - V_2, e^{-i\kappa\cdot x} \rangle \big| \lesssim
(1 + |\kappa|) \sum_{j,k=1}^2 \| V_j \|_{X^{-1/2}_{\zeta_k}} \sum_{l,m=1}^2 \| V_l \|_{X^{-1/2}_{\zeta_m}}.
\label{in:Fourier}
\end{equation}

Before going on, prove the boundedness of the operator multiplication by $V_1 - V_2$ from $X^{1/2}_{\zeta_1}(B_0)$ to $X^{-1/2}_{\zeta_2}$. To do so, let $V$ denote a potential of the form \eqref{id:V}, consider $w \in X^{1/2}_{\zeta_1}(B_0)$, and show that there exists a positive
$C = C(d,\Gamma, \alpha, V^0)$ such that
\begin{equation}
\| V w \|_{X^{-1/2}_{\zeta_2}} \leq C \| w \|_{X^{1/2}_{\zeta_1}(B_0)}.
\label{in:boundednessV}
\end{equation}
We will prove this boundedness by duality. Let $u \in X^{1/2}_{\zeta_1}$ denote an arbitrary extension of $w \in X^{1/2}_{\zeta_1} (B_0)$ and note that
\[\langle Vw, \phi \rangle = \int_{\R^d} V^0 u\phi \, + \int_\Gamma \alpha u\phi \, \dd\sigma.  \]
The first of these terms on the right-hand side can be easily bounded using H\"older's inequality and Haberman's embedding (see\cite{zbMATH06490961})
\[\Big| \int_{\R^d} V^0 u\phi \, \Big| \leq \| V^0 \|_{L^{d/2}} \| u \|_{L^{p_d}} \| \phi \|_{L^{p_d}} \lesssim 
\| V^0 \|_{L^{d/2}} \| u \|_{X^{1/2}_{\zeta_1}} \| \phi \|_{X^{1/2}_{\zeta_2}}. \]
The second term, can be rewritten as follows
\[ \int_\Gamma \alpha u\phi \, \dd\sigma = \int_{T^t\Gamma} \alpha_{\zeta_1+\zeta_2} u_{\zeta_1} \phi_{\zeta_2} \, \dd T^\ast \sigma \]
with
\begin{align*}
& \alpha_{\zeta_1+\zeta_2} (y) = e^{-i\Im (\zeta_1 + \zeta_2) \cdot Ty} \alpha(Ty), \\
& u_{\zeta_1}(y) = e^{i\Im \zeta_1 \cdot Ty} u(Ty),\\
& \phi_{\zeta_2}(y) = e^{i\Im \zeta_2 \cdot Ty} \phi(Ty),
\end{align*}
where $T \in SO(d)$ is so that $\Re \zeta_1 = \tau T e_d = -\Re\zeta_2$, $T^t \Gamma = \{ T^t x : x \in \Gamma \}$ and $T^\ast \sigma (E) = \sigma (TE)$ with $TE = \{ Tx : x \in E \}$. Thus, the Cauchy--Schwarz inequality, the theorem \ref{th:trace} and the lemma \ref{lem:Yast-Bourgain_spaces} ---in the sections \ref{sec:trace} and \ref{sec:others}--- imply that
\[\Big| \int_\Gamma \alpha u\phi \, \dd\sigma \Big| \lesssim
\| \alpha_\zeta \|_{L^\infty(T^t \Gamma)} \| u_\zeta \|_{\dot{Y}^{1/2}_\tau} \| \phi_\zeta \|_{\dot{Y}^{1/2}_\tau}. \]
Since $\widehat{u_{\zeta_1}} (\xi) = \widehat{u}(T\xi - \Im \zeta_1) $, $\widehat{\phi_{\zeta_2}} (\xi) = \widehat{\phi}(T\xi - \Im \zeta_2) $ and 
$|q_\tau (\xi)| = |p_{\zeta_j} (T\xi - \Im\zeta_j)|$ with $j = 1,2$, we have that
\begin{equation*}
\Big| \int_\Gamma \alpha u\phi \, \dd\sigma \Big| \lesssim
\| \alpha \|_{L^\infty(\Gamma)} \| |p_{\zeta_1}|^{1/2} \widehat{u} \|_{L^2} \| |p_{\zeta_2}|^{1/2} \widehat{\phi} \|_{L^2} \leq \| \alpha \|_{L^\infty(\Gamma)} \| u \|_{X^{1/2}_{\zeta_1}} \| \phi \|_{X^{1/2}_{\zeta_2}}.
\end{equation*}
Gathering the inequalities for $V^0$ and $\alpha \, \dd \sigma$, we obtain that
\[\langle Vw, \phi \rangle \lesssim (\| V^0 \|_{L^{d/2}} + \| \alpha \|_{L^\infty(\Gamma)}) \| u \|_{X^{1/2}_{\zeta_1}} \| \phi \|_{X^{1/2}_{\zeta_2}}, \]
and, consequently that
\[ \| Vw \|_{X^{-1/2}_{\zeta_2}} \lesssim (\| V^0 \|_{L^{d/2}} + \| \alpha \|_{L^\infty(\Gamma)}) \| u \|_{X^{1/2}_{\zeta_1}} \]
where $u$ is an arbitrary extension of $w$. Taking the infimum, between the norm of all the possible extensions of $w$, we get the inequality \eqref{in:boundednessV}.

We now go back to the inequality \eqref{in:Fourier}. Our aim is to 
show that its right-hand side tends to zero in some sense as $\tau$ 
in \eqref{id:zeta} goes to infinity. Due to the $\delta$-shell parts of $V_1$ and $V_2$, this decay will be possible in average as Haberman and Tataru showed for the Calder\'on problem in \cite{zbMATH06145493}.

\begin{lemma}\label{lem:averaging} \sl Let $V$ be a potential of the form of 
\eqref{id:V}. If $\zeta = \zeta(\tau, T) \in \C^d$ is so that
$\zeta \cdot \zeta = -\lambda$ with $ \Re \zeta = \tau T e_1$, then 
for every $s < 1/2$, we have that
\[ \frac{1}{M} \int_M^{2M} \int_{SO(d)} \| V \|^2_{X^{-1/2}_\zeta} \, \dd \mu(T) \, \dd \tau \lesssim
\left\{
\begin{aligned}
M^{-\frac{1}{2}} \| V^0 \|^2_{L^{d/2}} + M^{-s} \| \alpha \|^2_{L^\infty(\Gamma)} & & d = 3 \\
M^{-1} \| V^0 \|^2_{L^{d/2}} + M^{-s} \| \alpha \|^2_{L^\infty(\Gamma)} & & d \geq 4
\end{aligned}
\right.
\]
where the implicit constant depends on $R_0$, $\Gamma$ and $d$. The measure $\mu$ denotes the Haar measure on $SO(d)$.
\end{lemma}

\begin{proof}
Start by estimating the critically singular part of $V$. If $d \geq 4$,
\[ \| V^0 \|_{X^{-1/2}_{\zeta(\tau, T)}}  \leq \tau^{-1/2} \| V^0 \|_{L^2}  \lesssim \tau^{-1/2} \| V^0 \|_{L^{d/2}} \]
since $\supp V^0 \subset B_0$ and $d/2 \geq 2$ for $d \geq 4$. In the case $d=3$, by the dual inequality to Haberman's embedding (see \cite{zbMATH06490961}),
\[\| V^0 \|_{X^{-1/2}_{\zeta(\tau, T)}} \lesssim \tau^{-d(1/p_d^\prime - 1/q_d^\prime)} \| V^0 \|_{L^{q_d^\prime}} \lesssim \tau^{-1/4} \| V^0 \|_{L^{d/2}}\]
since $\supp V^0 \subset B_0$ and $d/2 \geq q_d^\prime$ for $d = 3$. This proves the part of the estimate corresponding to the critically singular component of $V^0$. We focus now on the $\delta$-shell component. By Lemma 5.2 in \cite{zbMATH06490961} we have that
\[\frac{1}{M} \int_M^{2M} \int_{SO(d)} \| V \|^2_{X^{-1/2}_{\zeta(\tau, T)}} \, \dd \mu(T) \, \dd \tau \lesssim M^{-1} \| L(\alpha \, \dd\sigma) \|_{\dot{H}^{-1/2}}^2 + \| H(\alpha \, \dd\sigma) \|_{\dot{H}^{-1}}^2 \]
where $\mathcal{F} H(\alpha \, \dd\sigma) = \mathbf{1}_{|\xi|\geq 2M} \mathcal{F} (\alpha \, \dd\sigma)$ and $ L(\alpha \, \dd\sigma) = \alpha \, \dd\sigma - H(\alpha \, \dd\sigma)$. Thus, for every $\varepsilon \in (0 , 1/2)$,  we have
\[\frac{1}{M} \int_M^{2M} \int_{SO(d)} \| V \|_{X^{-1/2}_{\zeta(\tau, T)}} \, \dd \mu(T) \, \dd \tau \lesssim M^{-1 + 2\varepsilon} \| \alpha \, \dd\sigma \|_{\dot{H}^{-1/2 - \varepsilon}}^2. \]
Since $\supp(\alpha \, \dd\sigma) \subset B_0$, we have
\[\| \alpha \, \dd\sigma \|_{\dot{H}^{-1/2 - \varepsilon}} \lesssim \| \alpha \, \dd\sigma \|_{H^{-1/2 - \varepsilon}}. \]
Using the dual of the usual trace theorem for Sobolev spaces we have that the right-hand side of last inequality is bounded by $ \| \alpha \|_{L^2(\Gamma)}$. Since $\Gamma$ is compact, we have that $\| \alpha \|_{L^2(\Gamma)} \lesssim \| \alpha \|_{L^\infty(\Gamma)} $, which proves the part of the inequality corresponding to the $\delta$-shell component of the potential.
\end{proof}

We will apply this lemma to show that $\langle V_1 - V_2, e^{-i\kappa\cdot x} \rangle = 0$. For that, consider $\zeta_1$ and $\zeta_2$ as in \eqref{id:zeta} with $\kappa/|\kappa| = T_\kappa e_d$, $\eta = T_\kappa S e_1$ and $\theta = T_\kappa S e_2$, where $S \in SO(d)$ such that $S e_d = e_d$. Identifying the set $\{ S \in SO(d) : S e_d = e_d \}$ with $SO(d - 1)$, we have that 
\begin{align*}
\big| \langle V_1 - V_2, & e^{-i\kappa\cdot x} \rangle \big| = \frac{1}{M} \int_M^{2M} \int_{SO(d - 1)} \big| \langle V_1 - V_2, e^{-i\kappa\cdot x} \rangle \big| \, \dd \mu(S) \, \dd \tau \\
& \lesssim
(1 + |\kappa|) \sum_{j,k,l,m = 1}^2 \frac{1}{M} \int_M^{2M} \int_{SO(d - 1)} \| V_j \|_{X^{-1/2}_{\zeta_k}} \| V_l \|_{X^{-1/2}_{\zeta_m}} \, \dd \mu(S) \, \dd \tau,
\end{align*}
where  $\zeta_j = \zeta_j (\tau, S)$.
Applying Cauchy--Schwarz in the integration with respect to $S$ and $\tau$, and the lemma \ref{lem:averaging}, we obtain that
\[\langle V_1 - V_2, e^{-i\kappa\cdot x} \rangle = 0 , \qquad \forall \kappa \in \R^d\]
after making $M$ tend to infinity. By the injectivity of the Fourier transform, we know that $V_1 = V_2$. This proves the theorem \ref{th:uniqueness}.

\section{Well-suited estimates for the resolvent}\label{sec:resolvent}
In this section we prove the lemmas that we used in the section 
\ref{sec:scattering} to derive the resolvent estimates in the spaces $X_\lambda^\ast$ and $X_\lambda$. Additionally, we derive as a 
consequence the classical resolvent estimates \eqref{in:AH-KPV} and 
\eqref{in:KRS}, together with some inequalities for the conjugated Laplacian 
---including \eqref{in:HT_localization} and \eqref{in:H}.

\subsection{The refined estimates}\label{sec:refined}
We start by stating a modification of \eqref{in:AH-KPV} that turns out to 
be better adapted for our goal. To do so, we need to introduce the 
spaces $Y_\lambda^\ast$, and to call the definition of $Y_\lambda$ given in the section \ref{sec:scattering}.
\begin{definition}\label{def:Ylambdaast}\sl Let $Y_\lambda^\ast$ be 
the set of $u\in\mathcal{S}^\prime(\R^d)$ so that
\begin{equation*}
\|m_\lambda^{1/2}\widehat{P_{<I} u}\|_{L^2}^2 + \sum_{k\in I} \lambda^{1/2} \vvvert P_k u \vvvert_\ast^2 +\sum_{k > k_\lambda + 1}\|m_\lambda^{1/2}\widehat{P_k u}\|_{L^2}^2 <\infty,
\end{equation*}
where $m_\lambda(\xi) = |\lambda-|\xi|^2|$.
For $u \in Y_\lambda^\ast$, define the norm
\[\|u\|_{Y_\lambda^\ast}^2 = \|m_\lambda^{1/2}\widehat{P_{<I} u}\|_{L^2}^2 + \sum_{k\in I} \lambda^{1/2} \vvvert P_k u \vvvert_\ast^2 +\sum_{k > k_\lambda + 1}\|m_\lambda^{1/2}\widehat{P_k u}\|_{L^2}^2. \]
\end{definition}
We now state the inequality, and prove it later.
\begin{lemma}\label{lem:Ylambda}\sl There exists a constant $C > 0$ only depending on $d$ so that
\[\| (\Delta + \lambda \pm i 0)^{-1} f \|_{Y_\lambda^\ast} \leq C \| f \|_{Y_\lambda}\]
for all $f\in \mathcal{S}(\R^d)$.
\end{lemma}

\begin{remark}\label{rem:Ylambda} The resolvent estimate in the lemma \ref{lem:Ylambda} is equivalent to \eqref{in:AH-KPV}, the fact that the latter inequality implies the former one is straightforward. The converse implication is proved in the Corollary \ref{cor:AH-KPV} in the section \ref{sec:classical}.
\end{remark}

We continue with our next refined estimate, which consists of a well-suited version of \eqref{in:KRS}. Again, we start by introducing the space $Z_\lambda^\ast$, and calling the definition of $Z_\lambda$ in the section \ref{sec:classical}.
\begin{definition}\label{def:Zlambdaast}\sl Let $Z_{\lambda, p}^\ast$ with $p \in [q_d, p_d]$ be 
the set of $u\in\mathcal{S}^\prime(\R^d)$ so that
\begin{equation*}
\|m_\lambda^{1/2}\widehat{P_{<I} u}\|_{L^2}^2 + \sum_{k\in I} \lambda^{d(\frac{1}{p}-\frac{1}{p_d})}\|P_k u\|_{L^p}^2 +\sum_{k> k_\lambda + 1}\|m_\lambda^{1/2}\widehat{P_k u}\|_{L^2}^2 <\infty,
\end{equation*}
where $m_\lambda(\xi) = |\lambda-|\xi|^2|$.
For $u \in Z_{\lambda, p}^\ast$ we define the norm
\[\|u\|_{Z_{\lambda, p}^\ast}^2 =\|m_\lambda^{1/2}\widehat{P_{<I} u}\|_{L^2}^2 + \sum_{k\in I} \lambda^{d(\frac{1}{p}-\frac{1}{p_d})}\|P_k u\|_{L^{p}}^2 +\sum_{k> k_\lambda + 1}\|m_\lambda^{1/2}\widehat{P_k u}\|_{L^2}^2. \]
For simplicity, we write $Z_\lambda^\ast$ instead of $Z_{\lambda,q_d}^\ast$.
\end{definition}
\begin{remark} By Bernstein's inequality
\[\|u\|_{Z_{\lambda, p_d}^\ast} \lesssim \|u\|_{Z_{\lambda, p}^\ast} \lesssim \|u\|_{Z_\lambda^\ast}, \]
and therefore
\[ Z_\lambda^\ast \subset Z_{\lambda, p}^\ast \subset Z_{\lambda, p_d}^\ast. \]
\end{remark}

\begin{lemma}\label{lem:Zlambda}\sl There exists a constant $C > 0$ only depending on $d$ so that
\[\| (\Delta + \lambda \pm i 0)^{-1} f \|_{Z_\lambda^\ast} \leq C \| f \|_{Z_\lambda}\]
for all $f\in \mathcal{S}(\R^d)$.
\end{lemma}

\begin{remark}\label{rem:Zlambda} The resolvent estimate in the lemma \ref{lem:Zlambda} is equivalent to \eqref{in:KRS}, the fact that the latter inequality implies the former one is straightforward. The converse implication is proved in the Corollary \ref{cor:KRS} in the section \ref{sec:classical}.
\end{remark}

\begin{proof}[Proof of the lemma \ref{lem:Zlambda}]
The inequality to be proved follows from \eqref{in:KRS} for $d \geq 3$. The case $d = 2$ 
was not considered in \cite{zbMATH04050093}. We include 
here an argument that does not require Stein's interpolation theorem,
which was the approach followed in \cite{zbMATH04050093}, and works for dimension 
$d = 2$.

Start by providing an explicit formula of $(\Delta + \lambda \pm i 0)^{-1}$:
\begin{equation}
\langle(\Delta+\lambda\pm i0)^{-1} f, \overline{g}\rangle=\lim_{\epsilon\to 0}\int_{m_\lambda>\epsilon}\frac{\widehat{f}(\xi)\overline{\widehat{g}(\xi)}}{\lambda-|\xi|^2} \, \dd \xi \mp i\frac{\pi}{2\lambda^{1/2}}\int_{S_\lambda}\widehat{f}(\xi)\overline{\widehat{g}(\xi)} \, \dd S_\lambda(\xi),
\label{id:resolvent}
\end{equation}
where $\dd S_\lambda$ stands for the volume form on $S_\lambda$. If $k > k_\lambda + 1$,
\[\langle  P_k (\Delta+\lambda\pm i0)^{-1} f, \overline{g}\rangle=\int_{\R^d}\frac{\widehat{P_k f}(\xi)\overline{\widehat{g}(\xi)}}{\lambda-|\xi|^2} \, \dd \xi\]
and consequently, $ \|m_\lambda^{1/2} \mathcal{F} P_k (\Delta+\lambda\pm i0)^{-1} f \|_{L^2} = \|m_\lambda^{-1/2}\widehat{P_k f}\|_{L^2} $, where $\mathcal{F}$ denotes the Fourier transform.
The same holds with the projector $P_{<I}$.
For the critical frequencies $k \in I$, the identity 
\eqref{id:resolvent} does not become simpler. Start by the second 
term. Re-scaling the integral to bring $S_\lambda$ back to $
\mathbb{S}^{d - 1}$, and then applying Cauchy--Schwarz we get
\begin{equation}
\Big| i\frac{\pi}{2\lambda^{1/2}}\int_{S_\lambda}\widehat{P_k f}(\xi)\overline{\widehat{g}(\xi)} \, \dd S_\lambda(\xi) \Big|
\lesssim \lambda^{d/2 - 1}  \| \widehat{(P_k f)_\lambda} \|_{L^2(\mathbb{S}^{d - 1} )}  \| \widehat{g_\lambda} \|_{L^2(\mathbb{S}^{d - 1} )},
\label{in:bilinear_Slambda}
\end{equation}
where $(P_k f)_\lambda (x) = \lambda^{-d/2} P_k f (x/\lambda^{1/2}) $ 
and $g_\lambda (x) = \lambda^{-d/2} g (x/\lambda^{1/2}) $. The 
restriction version of the Tomas--Stein theorem, together with an appropriate scale change,  yields
\begin{equation*}
\Big| i\frac{\pi}{2\lambda^{1/2}}\int_{S_\lambda}\widehat{P_k f}(\xi)\overline{\widehat{g}(\xi)} \, \dd S_\lambda(\xi) \Big| \lesssim \lambda^{-d/2 - 1} \lambda^{d/q_d^\prime} \| P_k f \|_{L^{q_d^\prime}}  \| g \|_{L^{q_d^\prime}}.
\end{equation*}
Since $\lambda^{-d/2 - 1} \lambda^{d/q_d^\prime} = \lambda^{d(1/q_d^\prime - 1/p_d^\prime)}= \lambda^{-d/2(1/q_d - 1/p_d)} \lambda^{d/2(1/q_d^\prime - 1/p_d^\prime)} $, the inequality for the second term 
of the right-hand side of \eqref{id:resolvent} follows by duality.
To prove the inequality for the first term, we introduce
\[\mathcal{P}_\lambda f (x) = \frac{1}{(2 \pi)^{d/2}} \lim_{\epsilon \to 0} \int_{m_\lambda > \epsilon} \frac{e^{i x \cdot \xi}}{\lambda-|\xi|^2} \widehat{f}(\xi) \, \dd \xi \]
since to finish the proof of this lemma is enough to show that
\begin{equation}
\lambda^{\frac{d}{2}(\frac{1}{q_d}-\frac{1}{p_d})} \| P_k \mathcal{P}_\lambda f \|_{L^{q_d}} \lesssim \lambda^{\frac{d}{2}(\frac{1}{q_d^\prime}-\frac{1}{p_d^\prime})} \| P_k f \|_{L^{q_d^\prime}}.
\label{in:Pk_KRS}
\end{equation}
We analyse $\mathcal{P_\lambda}$ by distinguishing the frequencies inside a neighbourhood of $S_\lambda$ of width 
$2 \delta \lambda^{1/2}$, from those outside. Consider
$\delta \in (0,1)$ ---to be chosen--- and set
\[\mathcal{P}_\lambda^1 f (x) = \frac{1}{(2 \pi)^{d/2}} \lim_{\epsilon \to 0} \int_{m_\lambda > \epsilon} \frac{e^{i x \cdot \xi}}{\lambda-|\xi|^2} \phi\Big( \frac{\lambda - |\xi|^2}{\lambda \delta} \Big) \widehat{f}(\xi) \, \dd \xi \]
and $\mathcal{P}_\lambda^2 f = \mathcal{P}_\lambda f - \mathcal{P}_\lambda^1 f$, with $\phi \in C^\infty_0(\R; [0,1])$ so that $\phi(t) = 1$
for all $t \in [-1, 1]$ and $\phi(t) = 0$ whenever $|t| \geq 2$. By Bernstein's inequality, Plancherel's identity and the fact that
\[\supp \Big[ 1 - \phi\Big( \frac{\lambda - |\centerdot|^2}{\lambda \delta} \Big) \Big] \subset \{ \xi \in \R^d : \lambda \delta \leq |\lambda - |\xi|^2| \},\]
we have that, for $p \geq 2$,
\begin{align*}
\| P_k \mathcal{P}_\lambda^2 f \|_{L^p} &\lesssim \lambda^{\frac{d}{2}(\frac{1}{2}-\frac{1}{p})} \| \mathcal{P}_\lambda^2 P_k f \|_{L^2}
\lesssim \lambda^{\frac{d}{2}(\frac{1}{2}-\frac{1}{p})} \lambda^{-1} \| \widehat{P_k f} \|_{L^2} \\
& \lesssim \lambda^{\frac{d}{2}(\frac{1}{p^\prime}-\frac{1}{p})} \lambda^{-1} \| P_k f \|_{L^{p^\prime}},
\end{align*}
where $p^\prime$ is the dual exponent of $p$. In the previous chain of inequalities, we used that $\lambda^{1/2} \simeq 2^k $ since $k \in I$. Therefore, we have that, for $p \geq 2$,
\begin{equation}
\| P_k \mathcal{P}_\lambda^2 f \|_{L^p} \lesssim \lambda^{\frac{d}{2}(\frac{1}{p^\prime}-\frac{1}{p})} \lambda^{-1} \| P_k f \|_{L^{p^\prime}}
\label{in:Plambda2}
\end{equation}
with $p^\prime$ its dual exponent. In particular, considering $p=q_d$, we have for $\mathcal{P}_\lambda^2$ the corresponding inequality to \eqref{in:Pk_KRS} since $\lambda^{d/2(1/q_d^\prime - 1/q_d) - 1} = \lambda^{d/2(1/q_d^\prime - 1/d - 1/q_d - 1/d)} = \lambda^{d/2(1/q_d^\prime - 1/p^\prime_d - 1/q_d + 1/p_d)} $. It remains to prove \eqref{in:Pk_KRS} for $\mathcal{P}_\lambda^1$.
Start by rescaling so that $P_k \mathcal{P}_\lambda^1 f (x) = \lambda^{d/2 - 1} \mathcal{P}_1^1 (P_k f)_\lambda (\lambda^{1/2} x)$, then it is enough to prove
\begin{equation}
\| \mathcal{P}_1^1 g \|_{L^{q_d}} \lesssim \| g \|_{L^{q_d^\prime}}.
\label{in:rescaledKRS}
\end{equation}
Covering the $2\delta$-width neighbourhood of $\mathbb{S}^{d - 1}$ 
with balls centred at points on $\mathbb{S}^{d - 1}$ and radius $2 \delta^{1/2}$, we can reduce the study to understand an operator of the form
\begin{equation}
\mathcal{Q} g (x) = \frac{1}{(2 \pi)^{d/2}} \lim_{\epsilon \to 0} \int_{m_1 > \epsilon} \frac{e^{i x \cdot \xi}}{1-|\xi|^2} \phi\Big( \frac{1 - |\xi|^2}{\delta} \Big) \psi\Big( \frac{e_d - \xi}{\delta^{1/2}} \Big) \widehat{g}(\xi) \, \dd \xi 
\label{id:2Bflatten}
\end{equation}
with $\psi \in C^\infty_0(\R^d; [0,1])$ so that $\psi(0)\neq 0$ and $\psi(\eta) = 0$ 
whenever $ |\eta| \geq 2$.
The reduction to understand
$\mathcal{Q}$ instead of $\mathcal{P}_1^1$ comes from the fact that,
we can construct a partition of unity subordinated 
to the covering made of such balls so that,
the latter can be written as a sum of operators $\{ \mathcal{Q}_l : l = 1 \dots N_\delta \}$ with 
these looking as the former after a rotation. Thus, if 
\eqref{in:rescaledKRS} holds for operators as $\mathcal{Q}$, then 
\eqref{in:rescaledKRS} is also valid for $\mathcal{P}_1^1$. Indeed,
\[\| \mathcal{P}_1^1 g \|_{L^{q_d}} \leq \sum_{l = 1}^{N_\delta} \| \mathcal{Q}_l g \|_{L^{q_d}} \lesssim \| g \|_{L^{q_d^\prime}} \]
where $N_\delta$ is the number of ball needed to cover the
$2\delta$-width neighbourhood of $\mathbb{S}^{d - 1}$. Therefore, in 
order to 
finish the proof we just need to prove that \eqref{in:rescaledKRS} 
holds for $\mathcal{Q}$. Observe that $\mathcal{Q} g = K \ast g$ with
\[ K(x) = \frac{1}{(2 \pi)^{d}} \lim_{\epsilon \to 0} \int_{m_1 > \epsilon} \frac{e^{i x \cdot \xi}}{1-|\xi|^2} \phi\Big( \frac{1 - |\xi|^2}{\delta} \Big) \psi\Big( \frac{e_d - \xi}{\delta^{1/2}} \Big) \, \dd \xi.\]
Write $1 - |\xi|^2 = (\Phi(\xi^\prime) + \xi_d)(\Phi(\xi^\prime) - \xi_d)$ with $\xi = (\xi^\prime, \xi_d)$ and $\Phi (\xi') = (1 - |\xi^\prime|^2)^{1/2} $, and take $\delta < 1/8$ so that
\[ \Phi(\xi^\prime) + \xi_d \geq 1 \qquad \forall \xi \in \supp \psi\Big( \frac{e_d - \centerdot}{\delta^{1/2}} \Big). \]
Changing variables, in the integrand defining the kernel $K$, according to $\eta^\prime = \xi^\prime$ and $\eta_d =  \Phi(\xi^\prime) - \xi_d$ we have that
\begin{equation}
K(x) = \frac{1}{(2 \pi)^{d - 1/2}} \int_{\R^{d - 1}} e^{i(x^\prime \cdot \eta^\prime + x_d \Phi(\eta^\prime))} a(\eta^\prime, x_d) \, \dd \eta^\prime
\label{id:kernel}
\end{equation}
with
\begin{equation}
a(\eta^\prime, x_d) = \frac{1}{(2 \pi)^{1/2}} \mathrm{p.v.} \int_{\R} \frac{e^{-ix_d \eta_d}}{\eta_d} \Psi(\eta) \, \dd \eta_d,
\label{id:fourierPV}
\end{equation}
where
\[\Psi(\eta) = \frac{1}{\Phi(\xi^\prime) + \xi_d} \phi\Big( \frac{1 - |\xi|^2}{\delta} \Big) \psi\Big( \frac{e_d - \xi}{\delta^{1/2}} \Big). \]
The term $a(\eta^\prime, x_d)$ reminds the well-known identity
\[ \frac{1}{(2\pi)^{1/2}} \mathrm{p.v.} \int_{\R} \frac{e^{-ist}}{t} \varphi(t) \, \dd t = -i \frac{1}{2}  \int_{\R} \sign (s - t) \widehat{\varphi} (t) \,\dd t, \]
and consequently,
\[ a(\eta^\prime, x_d) = -i \frac{1}{2}  \int_{\R} \sign (x_d - y_d) \mathcal{F}_d \Psi (\eta^\prime, y_d) \, \dd y_d, \]
where $\mathcal{F}_d$ stands for the $1$-dimensional Fourier transform applied to the last variable. Let us now get back to estimate $\mathcal{Q} g$. Note that, on the one hand
\begin{equation}
\begin{aligned}
\Big\| \int_{\R^{d - 1}} K(\centerdot - y^\prime, x_d - y_d) & g(y^\prime, y_d) \, \dd y^\prime \Big\|_{L^\infty(\R^{d - 1})} \\
& \leq \| K(\centerdot, x_d - y_d) \|_{L^\infty(\R^{d - 1})} \| g(\centerdot, y_d) \|_{L^1(\R^{d - 1})}.
\end{aligned}
\label{in:LinftyL1}
\end{equation}
On the other hand, Plancherel's identity applied to the first
$d-1$ variables yields
\begin{equation}
\begin{aligned}
\Big\| \int_{\R^{d - 1}} K(\centerdot - y^\prime, x_d - y_d) & g(y^\prime, y_d) \, \dd y^\prime \Big\|_{L^2(\R^{d - 1})} \\
& \lesssim \| \mathcal{F}^\prime K(\centerdot, x_d - y_d) \|_{L^\infty(\R^{d - 1})} \| g(\centerdot, y_d) \|_{L^2(\R^{d - 1})},
\end{aligned}
\label{in:L2L2}
\end{equation}
where $\mathcal{F}^\prime$ stands for the $(d-1)$-dimensional Fourier transform applied to the first $d-1$ variables. Furthermore, from the expression \eqref{id:kernel} one sees that
\[ \| \mathcal{F}^\prime K(\centerdot, x_d - y_d) \|_{L^\infty(\R^{d - 1})} \simeq \| a(\centerdot, x_d - y_d)) \|_{L^\infty(\R^{d - 1})} \lesssim  \int_{\R} \| \mathcal{F}_d \Psi (\centerdot, y_d) \|_{L^\infty(\R^{d - 1})} \, \dd y_d. \]
Interpolating the inequalities \eqref{in:LinftyL1} and \eqref{in:L2L2}, we get
\begin{equation}
\begin{aligned}
\Big\| \int_{\R^{d - 1}} K(\centerdot - y^\prime, x_d - y_d) & g(y^\prime, y_d) \, \dd y^\prime \Big\|_{L^q(\R^{d - 1})} \\
& \leq \| K(\centerdot, x_d - y_d) \|_{L^\infty(\R^{d - 1})}^{\frac{1}{q^\prime} - \frac{1}{q}} \| g(\centerdot, y_d) \|_{L^{q^\prime}(\R^{d - 1})}.
\end{aligned}
\label{in:LqdLqdprime}
\end{equation}
As a consequence of the stationary phase theorem (which exploits the curvature of $\mathbb{S}^{d - 1}$),
\[\| K(\centerdot, x_d - y_d) \|_{L^\infty(\R^{d - 1})} \lesssim (1 + |x_d - y_d|)^{- \frac{d - 1}{2}}.\]
Thus,
\[\|  \mathcal{Q} g \|_{L^q} \lesssim \Big\| \int_\R |x_d - y_d|^{- \frac{d - 1}{2}\big( \frac{1}{q^\prime} - \frac{1}{q} \big)} \| g(\centerdot, y_d) \|_{L^{q^\prime}(\R^{d - 1})} \, \dd y_d \Big\|_{L^q(\R_{x_d})}.\]
Considering $q = q_d$, we can apply the Hardy--Littlewood--Sobolev inequality and conclude that
\begin{equation*}
\| \mathcal{Q} g \|_{L^{q_d}} \lesssim \| g \|_{L^{q_d^\prime}}
\end{equation*}
holds, which was the last ingredient to finish the proof of the lemma \ref{lem:Zlambda}. 
\end{proof}

\begin{proof}[Proof of the lemma \ref{lem:Ylambda}] The wanted inequality follows from \eqref{in:AH-KPV}, however, we give here a simple proof for completeness. The argument follows the general scheme of the proof of the lemma \ref{lem:Zlambda} but simpler, since no interpolation is required, neither the curvature of $S_\lambda$ is exploited.

The estimate for the non-critical frequencies 
is straightforward, and works exactly as in the lemma 
\ref{lem:Zlambda}. To study the critical frequencies, we start 
analysing the second term on the right-hand side of 
\eqref{id:resolvent}, and obtain again the inequality \eqref{in:bilinear_Slambda}. Applying the trace theorem ---dual version of Theorem 7.1.26 in \cite{zbMATH01950198} (see also Theorem 14.1.1 in \cite{zbMATH02123716}), together with a change of scale,
we have
\begin{equation*}
\Big| i\frac{\pi}{2\lambda^{1/2}}\int_{S_\lambda}\widehat{P_k f}(\xi)\overline{\widehat{g}(\xi)} \, \dd S_\lambda(\xi) \Big| \lesssim \lambda^{-d/2 - 1} \lambda^{(d+1)/2} \vvvert P_k f \vvvert  \vvvert g \vvvert.
\end{equation*}
Since $\lambda^{-d/2 - 1} \lambda^{(d+1)/2} = \lambda^{-1/4} \lambda^{-1/4} $, the inequality for the second term 
of the right-hand side of \eqref{id:resolvent} follows by duality. To 
prove the inequality for the first term, we split again $\mathcal{P}_\lambda = \mathcal{P}_\lambda^1 + \mathcal{P}_\lambda^2 $. Note that 
using \eqref{in:Plambda2} with $p = 2d/(d - 1)$, we obtain
\begin{equation*}
\| P_k \mathcal{P}_\lambda^2 f \|_{L^2(D_j)} \lesssim 2^\frac{j}{2} \| P_k \mathcal{P}_\lambda^2 f \|_{L^p} \lesssim 2^\frac{j}{2} \lambda^{-\frac{1}{2}} \| P_k f \|_{L^{p^\prime}} \leq 2^\frac{j}{2} \lambda^{-\frac{1}{2}}  \sum_{l \in \N_0} \| P_k f \|_{L^{p^\prime}(D_l)}.
\end{equation*}
Hence, by H\"older's inequality
\begin{equation}
\lambda^{\frac{1}{4}} \vvvert P_k \mathcal{P}_\lambda^2 f \vvvert_\ast \lesssim \lambda^{-\frac{1}{4}} \vvvert P_k f \vvvert.
\label{in:P2lambdaY}
\end{equation}
We next prove \eqref{in:P2lambdaY} for $\mathcal{P}_\lambda^1$. After rescaling $P_k \mathcal{P}_\lambda^1 f (x) = \lambda^{d/2 - 1} \mathcal{P}_1^1 (P_k f)_\lambda (\lambda^{1/2} x)$, it is enough to prove
\begin{equation}
\sup_{j \in \N_0} \big( (\lambda^{1/2} 2^j)^{-1/2} \| \mathcal{P}_1^1 g \|_{L^2(\lambda^{1/2} D_j)} \big) \lesssim \sum_{j \in \N_0} (\lambda^{1/2} 2^j)^{1/2} \| g \|_{L^2(\lambda^{1/2} D_j)}.
\label{in:rescaledAH-KPV}
\end{equation}
As in the lemma \ref{lem:Zlambda}, the analysis can be reduced to study the operator $\mathcal{Q}$ in \eqref{id:2Bflatten}. In fact, we only need to check that \eqref{in:rescaledAH-KPV} holds for $\mathcal{Q}$. Indeed, applying \eqref{in:L2L2} and \eqref{id:kernel} we have
\begin{align*}
\| \mathcal{Q} g & \|_{L^2(\lambda^{1/2} D_j)} \\
& \leq \Big\| \int_\R \big\| \int_{\R^{d - 1}} K(\centerdot - y^\prime, x_d - y_d) g(y^\prime, y_d) \, \dd y^\prime \big\|_{L^2(\R^{d - 1})} \, \dd y_d \Big\|_{L^2(|x_d|<\lambda^{1/2} 2^j)} \\
& \lesssim \Big\| \int_\R \| g(\centerdot, y_d) \|_{L^2(\R^{d - 1})} \, \dd y_d \Big\|_{L^2(|x_d|<\lambda^{1/2} 2^j)} \\
& \lesssim (\lambda^{1/2} 2^j)^{1/2} \sum_{l \in \N_0} \int_\R \| [\mathbf{1}_{\lambda^{1/2} D_l} g](\centerdot, y_d) \|_{L^2(\R^{d - 1})} \, \dd y_d \\
& \lesssim (\lambda^{1/2} 2^j)^{1/2} \sum_{l \in \N_0} (\lambda^{1/2} 2^l)^{1/2} \| g \|_{L^2(\lambda^{1/2} D_l)},
\end{align*}
where $\mathbf{1}_{\lambda^{1/2} D_l}$ holds for the characteristic function of the set $\lambda^{1/2} D_l$. Therefore, \eqref{in:rescaledAH-KPV} holds for $\mathcal{Q}$ and the lemma \ref{lem:Ylambda} is proved.
\end{proof}

\subsection{The classical resolvent estimates}
\label{sec:classical}
The classical resolvent estimates \eqref{in:AH-KPV} and \eqref{in:KRS} 
follows from the lemmas \ref{lem:Ylambda} and \ref{lem:Zlambda} 
respectively, together with appropriate embeddings.

\begin{lemma} \label{lem:Lpembedding}\sl
For $p, q \in[q_d, p_d]$ with $q \leq p$, there exists a constant $C>0$ depending on $d$, $p$ and $q$ 
such that
\begin{equation*}
\lambda^{\frac{d}{2}(\frac{1}{p}-\frac{1}{p_d})}\|f\|_{L^p} \leq C \|f\|_{Z_{\lambda, q}^\ast} 
\end{equation*}
for all $f\in\mathcal{S}(\R^d)$.
\end{lemma}

\begin{proof}By the Littlewood--Paley theorem, Bernstein's inequalities and Plancherel 
identity, we have that
\begin{align*}
\|f  &\|_{L^p}^2\leq \|P_{<I} f\|_{L^p}^2 + \sum_{k\geq k_\lambda - 2}\|P_k f\|_{L^p}^2 \\
& \lesssim 2^{2k_\lambda d(\frac{1}{2}-\frac{1}{p})} \| \widehat{P_{<I} f}\|_{L^2}^2 + \sum_{k\in I} 2^{2kd(\frac{1}{q}-\frac{1}{p})}\|P_k f\|_{L^q}^2 +\sum_{k > k_\lambda + 1} 2^{2kd(\frac{1}{2}-\frac{1}{p})} \| \widehat{P_k f}\|_{L^2}^2.
\end{align*}
We have that
$ |\widehat{P_{<I} f}(\xi)| \simeq \lambda^{-1/2} m_\lambda(\xi)^{1/2} |\widehat{P_{<I} f}(\xi)|$ for all $\xi \in \R^d$. Hence,
\begin{equation}
\| \widehat{P_{<I} f}\|_{L^2}\simeq \lambda^{-1/2} \|m_\lambda^{1/2} \widehat{P_{<I} f}\|_{L^2}.
\label{in:low_frequencies}
\end{equation}
If $k>k_\lambda+1$, we have that $ |\widehat{P_k f}(\xi)| \simeq 2^{-k} m_\lambda(\xi)^{1/2} |\widehat{P_k f}(\xi)|$ for all $\xi \in \R^d$. Hence,
\begin{equation}
\| \widehat{P_k f}\|_{L^2}\simeq 2^{-k}\|m_\lambda^{1/2} \widehat{P_k f}\|_{L^2}.
\label{in:high_frequencies}
\end{equation}
Therefore,
\begin{align*}
\|f\|_{L^p}^2 
\lesssim & \sum_{k\in I} 2^{2kd(\frac{1}{q}-\frac{1}{p})}\|P_k f\|_{L^q}^2 +\lambda^{-1} 2^{2k_\lambda d(\frac{1}{2}-\frac{1}{p})}\|m_\lambda^{1/2}\widehat{P_{<I} f}\|_{L^2}^2\\
&+ 2^{-2k_\lambda d(\frac{1}{p}-\frac{1}{p_d})} \sum_{k>k_\lambda+1}\|m_\lambda^{1/2}\widehat{P_k f}\|_{L^2}^2.
\end{align*}
Since the critical scale $2^{k_\lambda}$ is of the order of $\lambda^{1/2}$, we have
\begin{align*}
\|f\|_{L^p}^2 \lesssim & \sum_{k\in I} \lambda^{d(\frac{1}{q}-\frac{1}{p})}\|P_k f\|_{L^q}^2 \\
& + \lambda^{- d(\frac{1}{p}-\frac{1}{p_d})} \Big( \|m_\lambda^{1/2}\widehat{P_{<I} f}\|_{L^2}^2 + \sum_{k > k_\lambda + 1} \|m_\lambda^{1/2}\widehat{P_k f}\|_{L^2}^2\Big).
\end{align*}
Finally, multiplying both sides by $\lambda^{d(1/p-1/p_d)}$ and taking square root, we obtain the embedding we were looking for.
\end{proof}

\begin{lemma} \label{lem:Lp*embedding} \sl For $p^\prime, q^\prime \in [p_d^\prime, q_d^\prime]$ with $p^\prime \leq q^\prime$, there exists a constant $C > 0$ depending on $d$, $p^\prime$ and $q^\prime$ such that
\[\| f \|_{Z_{\lambda, q^\prime}} \leq C \lambda^{\frac{d}{2}(\frac{1}{p^\prime} - \frac{1}{p_d^\prime})} \| f \|_{L^{p^\prime}},\] 
for all $f \in \mathcal{S}(\R^d)$.
\end{lemma}
\marnote{\color{red} To change.}
\begin{proof}
It follows from the lemma \ref{lem:Lpembedding} by a standard duality argument, since the Banach space $(Z_{\lambda, q^\prime}, \| \centerdot \|_{Z_{\lambda, q^\prime}})$ is reflexive, its dual can be identified with $(Z_{\lambda, q}^\ast, \| \centerdot \|_{Z_{\lambda, q}^\ast})$ and $\mathcal{S}(\R^d)$ is dense in the latter space ---see the lemma \ref{lem:Zspaces} in the appendix \ref{app:framework}.
\end{proof}

\begin{corollary} \label{cor:KRS} \sl For $p\in[q_d, p_d]$ with $d \geq 2$, there exists a constant $C>0$ depending on $d$ and $p$ such that
\begin{equation*}
\| (\Delta + \lambda \pm i 0)^{-1} f \|_{L^p} \lesssim \lambda^{\frac{d}{2}\big( \frac{1}{p^\prime} - \frac{1}{p} \big) - 1} \| f \|_{L^{p^\prime}}
\end{equation*}
for all $f\in\mathcal{S}(\R^d)$.
\end{corollary}
\begin{proof}
This is an immediate consequence of the lemmas \ref{lem:Lpembedding}, \ref{lem:Lp*embedding} and \ref{lem:Zlambda}.
\end{proof}

\begin{remark} The corollary \ref{cor:KRS} was stated in 
\cite{zbMATH04050093} for $d\geq 3$. This corollary also holds for $d = 2$ 
including the endpoint $p=p_2$.
\end{remark}

\begin{lemma}\label{lem:L2embedding}\sl
There exists a constant $C>0$ depending on $d$ such that 
\begin{equation*}
\lambda^{1/4} \vvvert f \vvvert_\ast \leq C\|f\|_{Y_\lambda^\ast},
\end{equation*}
for all $f\in\mathcal{S}(\R^d)$.
\end{lemma}
\begin{proof}
By the triangle inequality and extending the domain of integration from $D_j$ to $\R^d$, we have that 
\begin{align*}
\|f\|_{L^2(D_j)}&\leq \|P_{<I} f\|_{L^2(D_j)} + \sum_{k\geq k_\lambda - 2}\|P_k f\|_{L^2(D_j)}\\
&\lesssim \|P_{<I} f\|_{L^2} + \sum_{k\in I}\|P_k f\|_{L^2(D_j)} + \sum_{k > k_\lambda + 1}  \|P_k f\|_{L^2}.
\end{align*}
Multiplying by $2^{-j/2}$ with $j \in \N_0$ and using the equivalences \eqref{in:low_frequencies} and \eqref{in:high_frequencies}:
\begin{align*}
2^{-j/2}\|f\|_{L^2(D_j)}& \lesssim \sum_{k\in I}2^{-j/2}\|P_k f\|_{L^2(D_j)}+ \lambda^{-1/2} \|m_\lambda^{1/2} \widehat{P_{<I} f}\|_{L^2}\\
&+\sum_{k>k_\lambda+1} 2^{-k}\|m_\lambda^{1/2} \widehat{P_k f}\|_{L^2}.
\end{align*}
Since there are only four critical frequencies, on has
\[\sum_{k\in I}2^{-j/2}\|P_k f\|_{L^2(D_j)} \simeq \Big(\sum_{k\in I}2^{-j}\|P_k f\|_{L^2(D_j)}^2 \Big)^{1/2}. \]
Using the Cauchy--Schwarz inequality, we can proceed as follows:
\[ \sum_{k>k_\lambda+1} 2^{-k}\|m_\lambda^{1/2} \widehat{P_k f}\|_{L^2} \lesssim 2^{-k_\lambda} \Big(\sum_{k>k_\lambda+1} \|m_\lambda^{1/2} \widehat{P_k f}\|_{L^2}^2\Big)^{1/2}. \]
The fact that the critical scale $2^{k_\lambda}$ is of the order of $\lambda^{1/2}$ implies that, after taking square, we obtain
\begin{align*}
\lambda^{1/2} 2^{-j}\|f\|_{L^2(D_j)}^2 \lesssim  & \, \lambda^{-1/2} \|m_\lambda^{1/2} \widehat{P_{<I} f}\|_{L^2}^2 \\
& + \sum_{k\in I}\lambda^{1/2}2^{-j}\|P_k f\|_{L^2(D_j)}^2 + \lambda^{-1/2} \sum_{k > k_\lambda + 1} \|m_\lambda^{1/2} \widehat{P_k f}\|_{L^2}^2.
\end{align*}
Taking the corresponding supremum of $j \in \N_0$, we obtain the embedding stated in the lemma.
\end{proof}

\begin{lemma} \label{lem:L2*embedding} \sl There exists a constant $C > 0$ depending on $d$ such that
\[\|f\|_{Y_\lambda} \leq C \lambda^{-1/4} \vvvert f \vvvert\]
for all $f \in \mathcal{S}(\R^d)$.
\end{lemma}
\begin{proof}
This lemma is consequence of the lemma \ref{lem:L2embedding} together with a duality argument that uses that $(Y_\lambda^\ast, \| \centerdot \|_{Y_\lambda^\ast})$ is the dual of $(Y_\lambda, \| \centerdot \|_{Y_\lambda})$, and the density of $\mathcal{S}(\R^d)$ in the former space. This duality argument is based on the Hahn--Banach theorem (see the corollary 1.4 in \cite{zbMATH05633610}).
\end{proof}

\begin{corollary} \label{cor:AH-KPV} \sl There exists a constant $C > 0$ depending on $d$ such that
\[\lambda^{1/2} \vvvert (\Delta + \lambda \pm i 0)^{-1} f \vvvert_\ast \leq C \vvvert f \vvvert\]
for all $f \in \mathcal{S}(\R^d)$.
\end{corollary}
\begin{proof}
This is an immediate consequence of the lemmas \ref{lem:L2embedding}, \ref{lem:L2*embedding} and \ref{lem:Ylambda}.
\end{proof}

\subsection{A trace theorem}\label{sec:trace} In this section we prove a trace theorem for the space $Y_\lambda^\ast$. This is an essential piece to construct the scattering solution for critically-singular and $\delta$-shell potentials, specially to deal with the $\delta$-shell component.

\begin{theorem}\label{th:trace} \sl Let $\Gamma$ be a compact hypersurface locally described by the graphs of Lipschitz functions. There exists a constant $C > 0$ only depending on $d$ and $\Gamma$ such that
\[\| f \|_{L^2(\Gamma)} \leq C \| f \|_{Y_\lambda^\ast}\]
for all $f \in \mathcal{S}(\R^d)$ and all $\lambda > 1$.
\end{theorem}
\begin{proof}
We first introduce a localization function denoted by $\chi$, which is not compactly supported. To do so, let
$\phi \in \mathcal{S}(\R^d)$ be a $[0,1]$-valued function 
so that its support is 
contained in $ \{ \xi \in \R^d: |\xi| \leq 1 \}$ and it is not 
identically zero.

Then, there exists $\delta \in (0, 1]$ such that
\[ \Big| \int_{\R^d} e^{i x \cdot \xi} \phi(\xi) \, \dd \xi \Big| \geq  \frac{1}{2} \int_{\R^d} \phi(\xi) \, \dd \xi > 0 \]
whenever $|x| \leq  \delta$. Let $\chi \in \mathcal{S}(\R^d)$ be defined by
\begin{equation}
\chi(x) = \frac{1}{(2 \pi)^{d /2}} \int_{\R^d} e^{i \delta/R x \cdot \xi} \phi(\xi) \, \dd \xi,
\label{id:chi}
\end{equation}
with $R \geq 1$ so that
$\Gamma \subset B = \{ x\in \R^d : |x| < R \}$.
Note that $|\chi (x)| \gtrsim 1$ whenever $|x| \leq R$ and $\supp \widehat{\chi} \subset \{ \xi \in \R^d: |\xi| \leq 1 \}$. Since $\Gamma$ is contained in $B$,
\[ \| f\|_{L^2(\Gamma)} \lesssim \|\chi f\|_{L^2(\Gamma)} \lesssim \sum_{l\in\Z} 2^{l/2}\|P_l (\chi f)\|_{L^2}. \]
In the last inequality we have used the trace theorem ---a
Besov-space form of Theorem 14.1.1 in \cite{zbMATH02123716}.
We now show that
\begin{equation}
\sum_{l\in\Z} 2^{l/2}\|P_l (\chi f)\|_{L^2} \lesssim \| f \|_{Y_\lambda^\ast}.
\label{in:embeddingCHI}
\end{equation}
Start considering the low frequencies $l \leq k_\lambda + 4$. The 
continuity of $P_l$ in $L^2(\R^d)$ and the fact that the sum of low 
frequencies is at most of the order of $2^{k_\lambda/2}$ imply that
\begin{equation}
\sum_{l \leq k_\lambda + 4} 2^{l/2}\|P_l (\chi f)\|_{L^2} \lesssim 2^{k_\lambda/2}\|\chi f\|_{L^2} \simeq \lambda^{1/4}\|\chi f\|_{L^2}.
\label{in:lowfrequencies}
\end{equation}
On the other hand,
\begin{equation}
\|\chi f\|_{L^2} = \Big( \sum_{j \in \N_0} \| \chi f \|_{L^2(D_j)}^2 \Big)^{1/2} \leq \Big( \sum_{j \in \N_0} 2^j \| \chi \|_{L^\infty(D_j)}^2 \Big)^{1/2} \vvvert f \vvvert_\ast.
\label{in:multiplication}
\end{equation}
Since $ \sup_{x \in \R^d} (1 + |x|)^{-N} |\chi(x)| < \infty $ for any 
$N \in \N$, the series on the right-hand side of the 
inequality in \eqref{in:multiplication} converges. Thus, 
\eqref{in:lowfrequencies}, \eqref{in:multiplication} and the lemma 
\ref{lem:L2embedding} shows that
\begin{equation}
\sum_{l \leq k_\lambda + 4} 2^{l/2}\|P_l (\chi f)\|_{L^2} \lesssim \lambda^{1/4}\vvvert f \vvvert_\ast \lesssim \| f \|_{Y_\lambda^\ast}.
\label{es:lowFREQchi}
\end{equation}
Finally, we consider the high frequencies $l > k_\lambda  + 4$. By the triangle inequality,
\begin{equation}
\sum_{l > k_\lambda + 4} 2^{l/2} \|P_l (\chi f)\|_{L^2} \leq \sum_{l > k_\lambda + 4} \sum_{k \in \Z} 2^{l/2}\|P_l (\chi P_k f)\|_{L^2}.
\label{in:triangle}
\end{equation}
Since the support of $\widehat{\chi}$ is contained $ \{ \xi \in \R^d: |\xi| \leq 1 \}$, we have that
\begin{equation}
\supp \widehat{\chi P_k f} \subset
\left\{ 
\begin{array}{l l}
\{ \xi \in \R^d : |\xi| \leq 2^3 \} & \textnormal{if}\, k < 2, \\
\{ \xi \in \R^d : 2^{k - 2} \leq |\xi| \leq 2^{k+2} \} & \textnormal{if}\, k \geq 2. 
\end{array}
\right.
\label{in:disjointness}
\end{equation}
Thus, $P_l (\chi P_k f) = 0$ whenever $k < 2$ and $l > 4$, or 
whenever $k \geq 2$ and $|l - k|> 3$. Consequently, the sum on the 
right-hand side of \eqref{in:triangle} only has the terms $l > k_\lambda + 4$ and $k \geq 2$ ---if $k<2$ the non-zero terms satisfy $l \leq 4$, but there are no $l$ satisfying $ k_\lambda + 4 < l \leq 4$ with $\lambda > 1$. Therefore,

\begin{align*}
\sum_{l > k_\lambda + 4} 2^{l/2} \|P_l (\chi f)\|_{L^2} &\leq \sum_{l > k_\lambda + 4} \sum_{|k -l|\leq 3} 2^{l/2}\|P_l (\chi P_k f)\|_{L^2}\\
& = \sum_{k > k_\lambda + 1} \sum_{|l - k|\leq 3} 2^{l/2}\|P_l (\chi P_k f)\|_{L^2} \lesssim  \sum_{k > k_\lambda + 1} 2^{k/2}\|P_k f\|_{L^2}.
\end{align*}
In the last inequality we used the continuity in $L^2(\R^d)$ of the operators $P_l$ and multiplication by $\chi$, and the fact that $\sum_{|l - k|< 3} 2^{l/2} \simeq 2^{k/2}$.
Then, by Plancherel's identity, \eqref{in:high_frequencies} and Cauchy--Schwarz applied to the sum, we obtain
\begin{equation*}
\sum_{l > k_\lambda + 4} 2^{l/2} \|P_l (\chi f)\|_{L^2} \lesssim \sum_{k > k_\lambda + 1} 2^{-k/2} \|m_\lambda^{1/2} \widehat{P_k f}\|_{L^2} \lesssim 2^{-k_\lambda/2} \| f \|_{Y_\lambda^\ast}.
\end{equation*}
This inequality, together with \eqref{es:lowFREQchi}, shows that \eqref{in:embeddingCHI} holds, and consequently
the theorem is proved.
\end{proof}

\begin{remark} The novelty of this trace theorem bases on showing that the operator multiplication by $\chi$, defined as in \eqref{id:chi}, is bounded from $Y^\ast_\lambda$ to $\dot{B}^{1/2}_{2, 1}(\R^d)$ with a norm independent of $\lambda$. Our next step will be to show that such an operator is in fact compact.
\end{remark}

\begin{lemma}\label{lem:compact} \sl Let $\chi$ be as in \eqref{id:chi} and $\lambda > 1$. 
Multiplication by $\chi$ defines a compact operator from $Y^\ast_\lambda$ to
$\dot{B}^{1/2}_{2, 1}(\R^d)$.
\end{lemma}

\begin{proof}
In order to prove the compactness of the operator multiplication by $\chi$, 
we will consider a bounded sequence $\{ u_n : n \in \N \} $ in
$ Y_\lambda^\ast$ and show that there exist a subsequence
$\{ u_{n(m)} : m \in \N \}$ and $u \in Y_\lambda^\ast$ so that
\[\lim_{m \to \infty} \| \chi u_{n(m)} - \chi u \|_{\dot{B}^{1/2}_{2, 1}} = 0.\]
We will show in the appendix \ref{app:framework} that  $Y^\ast_\lambda$ is 
the dual space of $Y_\lambda$ (see the lemma \ref{lem:Yspaces}). Thus, given 
a bounded sequence $\{ u_n : n \in \N \} $ in $ Y_\lambda^\ast$, we know by 
the Banach--Alaoglu--Bourbaki theorem that there exist a subsequence
$\{ u_{n(m)} : m \in \N \}$ and $u \in Y_\lambda^\ast$ so that
\begin{equation}
\lim_{m \to \infty} \langle u_{n(m)} - u, f \rangle = 0
\label{id:weak_star_convergence}
\end{equation}
for all $f \in Y_\lambda$. Here $\langle \centerdot, \centerdot \rangle $ 
stands for the duality pairing between $Y^\ast_\lambda$ and $Y_\lambda$. For 
convenience, let $v_m$ denote the difference $u_{n(m)} - u$. We will show that
\[ \lim_{m \to \infty} \sum_{l \in \Z} 2^{l/2} \| P_l(\chi v_m) \|_{L^2} = 0. \]
To do so, we will use the dominate convergence theorem (DCT for 
short), which could be applied after we shown that, for every $l \in \Z$, $\|P_l (\chi v_m)\|_{L^2} $ tends 
to $ 0$ as $m$ goes to infinity, and
\begin{equation}
2^{l/2} \|P_l (\chi v_m)\|_{L^2} \lesssim \lambda^{-1/4} 2^{l/2} \mathbf{1}_{l \leq k_\lambda + 4} + 2^{-l/2} \mathbf{1}_{l > k_\lambda + 4},
\label{in:boundedl1}
\end{equation}
where the implicit constant does not depend on $m$ and, $\mathbf{1}_{l \leq k_\lambda + 4}$ and $\mathbf{1}_{l > k_\lambda + 4}$ stand 
for the characteristic functions of the set $\{ l \in \Z:  l \leq k_\lambda + 4\}$ and $\{ l\in \Z: l > k_\lambda + 4 \}$. Note that we can apply the DCT because the sequence on the right-hand side of 
\eqref{in:boundedl1} belongs to $l^1(\Z)$. 

Let us first check that \eqref{in:boundedl1} holds. Start by analysing the case $l \leq k_\lambda + 4$. The boundedness of $P_l$ in $L^2(\R^d)$, the inequality \eqref{in:multiplication} and the lemma \ref{lem:L2embedding} implies that
\begin{equation}
2^{l/2} \|P_l (\chi v_m)\|_{L^2} \lesssim \lambda^{-1/4} 2^{l/2} \|v_m \|_{Y_\lambda^\ast}.
\label{in:lowFREQ}
\end{equation}
This inequality is only useful if $l \leq k_\lambda + 4$. Continue now with the case $l > k_\lambda + 4$. Using \eqref{in:disjointness} for $l > k_\lambda + 4$, the boundedness of $P_l$ and multiplication by $\chi$ in $L^2(\R^d)$, Plancherel's identity and \eqref{in:high_frequencies}, we have that
\begin{equation}
\begin{aligned}
2^{l/2} \|P_l (\chi v_m)\|_{L^2} &\leq 2^{l/2} \sum_{|k - l|\leq 3} \|P_l (\chi P_k v_m)\|_{L^2} \lesssim 2^{l/2} \sum_{|k - l|\leq 3} \| P_k v_m\|_{L^2} \\
& \simeq 2^{l/2} \sum_{|k - l|\leq 3} 2^{-k} \| m_\lambda^{1/2} \widehat{P_k v_m}\|_{L^2} \lesssim 2^{-l/2} \| v_m \|_{Y_\lambda^\ast}.
\end{aligned}
\label{in:highFREQ}
\end{equation}
The inequalities \eqref{in:lowFREQ} and \eqref{in:highFREQ}, together with the fact that $\{ v_m : m \in \Z \}$ is bounded in $Y^\ast_\lambda$, yields \eqref{in:boundedl1}.

It remains to prove that
\[\lim_{m \to \infty} \|P_l (\chi v_m)\|_{L^2} = 0. \]
We will show this using the DCT again. Start by checking the 
point-wise convergence:
\[P_l (\chi v_m)(x) = \frac{2^{ld}}{(2\pi)^{d/2}} \langle v_m, \chi \widehat{\psi} ( 2^l (x - \centerdot)) \rangle\]
where $\langle \centerdot, \centerdot \rangle $ 
stands for the duality pairing between $Y^\ast_\lambda$ and $Y_\lambda$, and $\psi$ denotes the base function used to construct the 
Littlewood--Paley projectors. Since $\chi \widehat{\psi} ( 2^l (x - \centerdot))$ belongs to $Y_\lambda$ for all $x\in \R^d$, the 
convergence \eqref{id:weak_star_convergence} implies that
\[\lim_{m\to \infty} P_l (\chi v_m)(x) = 0 \]
for all $x \in \R^d$. Continue with the domination:
\[|P_l (\chi v_m)(x)| \lesssim 2^{ld} \| v_m \|_{Y_\lambda^\ast} \|\chi \widehat{\psi} ( 2^l (x - \centerdot))  \|_{Y_\lambda}. \]
Note that, since $\{ v_m : m \in \Z \}$ is bounded in $Y^\ast_\lambda$, it is enough to see that the function $x \mapsto \|\chi \widehat{\psi} ( 2^l (x - \centerdot))  \|_{Y_\lambda} $ belongs to $L^2(\R^d)$. We finish the proof of this lemma showing that this is the case. By the lemma \ref{lem:L2*embedding}, and then using Cauchy--Schwarz, we know that
\begin{align*}
\|\chi \widehat{\psi} ( 2^l (x - \centerdot)) \|_{Y_\lambda}
& \lesssim \lambda^{-1/4} \sum_{j \in \N_0} 2^{j/2} \| \chi \widehat{\psi} ( 2^l (x - \centerdot)) \|_{L^2(D_j)} \\
&\lesssim \| (\sum_{j \in \N_0} 2^{cj} \mathbf{1}_{D_j} )^{1/2} \chi \widehat{\psi} ( 2^l (x - \centerdot)) \|_{L^2}
\end{align*}
where $c$ is a constant so that  $c > 1$. Consequently,
\begin{align*}
\int_{\R^d} \|\chi \widehat{\psi} ( 2^l (x - \centerdot)) \|_{Y_\lambda}^2 \, \dd x &\lesssim \int_{\R^d} \| (\sum_{j \in \N_0} 2^{cj} \mathbf{1}_{D_j} )^{1/2} \chi \widehat{\psi} ( 2^l (x - \centerdot)) \|_{L^2}^2 \, \dd x\\
& = \| ( \sum_{j \in \N_0} 2^{cj} |\chi|^2 \mathbf{1}_{D_j} ) \ast |\widehat{\psi} ( 2^l \centerdot)|^2 \|_{L^1}\\
& \leq \| \sum_{j \in \N_0} 2^{cj} |\chi|^2 \mathbf{1}_{D_j} \|_{L^1}  \| \widehat{\psi} ( 2^l \centerdot) \|_{L^2}^2.  
\end{align*}
Since $\chi$ and $\psi$ are in $\mathcal{S}(\R^d)$, the right-hand 
side of the previous chains of inequalities is bounded, which 
concludes the proof of this lemma.
\end{proof}

\subsection{Other estimates}\label{sec:others} In this section we state and derive several consequences from the embeddings and inequalities proved in the sections \ref{sec:refined} and \ref{sec:classical}. In particular, \eqref{in:HT_localization}, \eqref{in:H} and \eqref{in:KRS_CGO}, beside a scale invariant version of the Sylvester--Uhlmann inequality.

\begin{lemma}\label{lem:Yast-Bourgain_spaces}\sl Whenever $d \geq 3 $, there exists a constant $C>0$ depending on $d$ so that
\[\| f \|_{Y_\lambda^\ast} \leq C \| f \|_{\dot{Y}_\tau^{1/2}}\]
for all $f \in \mathcal{S}(\R^d)$ and $\lambda = \tau^2$.
\end{lemma}
\begin{proof}
The fact that $m_\lambda(\xi) \leq |q_\tau(\xi)|$ for all $\xi \in \R^d$ implies that
\[\|f\|_{Y_\lambda^\ast}^2 \leq \||q_\tau|^{1/2}\widehat{P_{<I} f}\|_{L^2}^2 + \sum_{k\in I} \lambda^{1/2} \vvvert P_k f \vvvert_\ast^2 +\sum_{k > k_\lambda + 1}\||q_\tau|^{1/2}\widehat{P_k f}\|_{L^2}^2. \]
Thus, if we prove that for $k \in I$ we have
\begin{equation}
\tau^{1/2} \vvvert P_k f \vvvert_\ast \lesssim \||q_\tau|^{1/2}\widehat{P_k f}\|_{L^2},
\label{in:critical_frequencies}
\end{equation}
then the result follows since there exists a constant $c > 0$ so that
\[|\widehat{P_{<I} f}(\xi)|^2 + \sum_{k \geq k_\lambda - 2} |\widehat{P_k f}(\xi)|^2 \leq c |\widehat{f}(\xi)|^2, \qquad \forall \xi \in \R^d. \]
Last inequality is a known property of the Littlewood--Paley 
projectors. To finish the proof of this lemma, we show that 
\eqref{in:critical_frequencies} holds. Let $g$ be defined by
\begin{equation}
\widehat{g}(\xi) = |q_\tau(\xi)|^{1/2} \widehat{f}(\xi),\, \forall \xi \in \R^d.
\label{id:aux_g}
\end{equation}
Thus, using the inversion formula of the Fourier transform and 
changing variables to $\rho = |\xi^\prime|$ and
$ \theta = \xi^\prime /|\xi^\prime| $ with
$\xi = (\xi^\prime, \xi_d) \in \R^{d-1} \times \R$, we have that
\begin{equation}
P_k f (x) = \frac{1}{(2\pi)^{d/2}} \int_\R e^{i x_d \xi_d} \int_0^\infty \int_{\mathbb{S}^{d - 2}} e^{i\rho x^\prime \cdot \theta} \frac{\widehat{P_kg}(\rho\theta, \xi_d)}{|q_\tau (\rho \theta, \xi_d)|^{1/2}} \, \dd S(\theta) \rho^{d - 2} \, \dd \rho \, \dd \xi_d,
\label{id:pkf}
\end{equation}
where $\dd S$ denotes the volume form on $\mathbb{S}^{d - 2}$. Note 
that, extending the integration from $D_j$ to $\{ x\in \R^d: |x^\prime| \leq 2^j\}$, applying Plancherel's identity in the variable $x_d$ and then Minkowski's inequality, we have
\begin{equation}
\begin{aligned}
& 2^{-j/2} \| P_k f \|_{L^2(D_j)} \\
& \, \lesssim 2^{-j/2} \Big( \int_{B_j^\prime} \big\| \int_0^\infty \int_{\mathbb{S}^{d - 2}} e^{i\rho x^\prime \cdot \theta} \frac{\widehat{P_kg}(\rho\theta, \centerdot)}{|q_\tau (\rho \theta, \centerdot)|^{1/2}} \, \dd S(\theta) \rho^{d - 2} \, \dd \rho  \big\|_{L^2(\R)}^2 \, \dd x^\prime \Big)^\frac{1}{2}\\
& \, \lesssim 2^{-j/2} \Big\| \int_0^\infty \Big( \int_{B_j^\prime} \big| \int_{\mathbb{S}^{d - 2}} e^{i\rho x^\prime \cdot \theta} \frac{\widehat{P_kg}(\rho\theta, \centerdot)}{|q_\tau (\rho \theta, \centerdot)|^{1/2}} \, \dd S(\theta) \big|^2 \, \dd x^\prime \Big)^\frac{1}{2}  \rho^{d - 2} \, \dd \rho \Big\|_{L^2(\R)},
\end{aligned}
\label{in:pre_EXTtrace}
\end{equation}
where $B_j^\prime = \{ x^\prime \in \R^{d - 1} : |x^\prime| < 2^j \} $. Next we change variables $\rho x^\prime = y^\prime$ so that
\begin{align*}
2^{-j/2} \Big( \int_{B_j^\prime} \big| \int_{\mathbb{S}^{d - 2}} & e^{i\rho x^\prime \cdot \theta}  \frac{\widehat{P_kg}(\rho\theta, \centerdot)}{|q_\tau (\rho \theta, \centerdot)|^{1/2}} \, \dd S(\theta) \big|^2 \, \dd x^\prime \Big)^\frac{1}{2} \\
& = \rho^{1 - d/2} (\rho 2^j)^{-1/2} \Big( \int_{\rho B_j^\prime} \big| \int_{\mathbb{S}^{d - 2}} e^{i y^\prime \cdot \theta} \frac{\widehat{P_kg}(\rho\theta, \centerdot)}{|q_\tau (\rho \theta, \centerdot)|^{1/2}} \, \dd S(\theta) \big|^2 \, \dd y^\prime \Big)^\frac{1}{2}.
\end{align*}
Applying the extension version of the trace theorem ---Theorem 7.1.26 in \cite{zbMATH01950198} valid here for $d \geq 3$--- we have that the right-hand side of the previous identity can be bounded so that the inequality \eqref{in:pre_EXTtrace} becomes
\begin{equation}
2^{-\frac{j}{2}} \| P_k f \|_{L^2(D_j)} \lesssim \Big\| \int_0^\infty \rho^{1 - \frac{d}{2}} \Big( \int_{\mathbb{S}^{d - 2}}  \frac{|\widehat{P_kg}(\rho\theta, \centerdot)|^2}{|q_\tau (\rho \theta, \centerdot)|} \, \dd S(\theta) \Big)^\frac{1}{2}  \rho^{d - 2} \, \dd \rho \Big\|_{L^2(\R)}.
\label{in:post-EXTtrace}
\end{equation}
Note that $|q_\tau (\rho \theta, \xi_d)|^2 = |\tau^2 - \rho^2 - \xi_d^2|^2 + |2 \tau \xi_d|^2$, which does not depend on $\theta$. 
Hence, by the Cauchy--Schwarz inequality applied to the integration 
in $\rho$, we have that the right-hand side of the previous 
inequality is bounded by a constant multiple of
\[\sup_{|\xi_d| \leq 2^{k + 1}} \Big( \int_0^{2^{k+1}} |q_\tau (\rho \theta, \xi_d)|^{-1} \, \dd \rho \Big)^\frac{1}{2} \Big\| \int_0^\infty \int_{\mathbb{S}^{d - 2}}  |\widehat{P_kg}(\rho\theta, \centerdot)|^2 \, \dd S(\theta) \rho^{d - 2} \, \dd \rho \Big\|_{L^2(\R)}. \]
Since $k \in I$, one can check that this term is bounded by $\tau^{-1/2} \| \widehat{P_kg} \|_{L^2}$. Thus, the inequality \eqref{in:post-EXTtrace} becomes
\[ 2^{-j/2} \| P_k f \|_{L^2(D_j)} \lesssim \tau^{-1/2} \| |q_\tau|^{1/2} \widehat{P_kf} \|_{L^2}. \]
Taking supremum in $j\in \N_0$ we obtain \eqref{in:critical_frequencies}. 
\end{proof}

\begin{corollary}[Haberman--Tataru] \sl Whenever $d \geq 3 $, there exists a constant $C>0$ depending on $d$ so that, if $\chi \in \mathcal{S}(\R^d)$, then
\[\tau^{1/2} \| \chi f \|_{L^2} \leq C \| f \|_{\dot{Y}_\tau^{1/2}}\]
for all $f \in \mathcal{S}(\R^d)$.
\end{corollary}
\begin{proof}
This is a consequence of the inequality \eqref{in:multiplication} and the lemmas \ref{lem:L2embedding} and \ref{lem:Yast-Bourgain_spaces}.
\end{proof}

\begin{lemma}\label{lem:Y-Bourgain_spaces}\sl Whenever $d \geq 3 $, there exists a constant $C>0$ depending on $d$ so that
\[\| f \|_{\dot{Y}_\tau^{-1/2}} \leq C \| f \|_{Y_\lambda}\]
for all $f \in \mathcal{S}(\R^d)$ and $\lambda = \tau^2$.
\end{lemma}
\begin{proof}
It follows from the lemma \ref{lem:Yast-Bourgain_spaces} by a standard duality argument.
\end{proof}

\begin{corollary}[Sylvester--Uhlmann]\sl Whenever $d \geq 3 $, there exists a constant $C>0$ depending on $d$ so that
\[ \vvvert (\Delta + 2 \tau \partial_{x_d} + \tau^2)^{-1} f \vvvert_\ast \leq \frac{C}{\tau} \vvvert f \vvvert\]
for all $f \in \mathcal{S}(\R^d)$.
\end{corollary}
\begin{proof}
It follows from the identity \eqref{id:isometry} and the lemmas \ref{lem:Yast-Bourgain_spaces}, \ref{lem:Y-Bourgain_spaces}, \ref{lem:L2embedding} and \ref{lem:L2*embedding}.
\end{proof}

\begin{lemma}\label{lem:Zast-Bourgain_spaces}\sl Whenever $d \geq 3 $, there exists a constant $C>0$ depending on $d$ so that
\[\| f \|_{Z_{\lambda, p_d}^\ast} \leq C \| f \|_{\dot{Y}_\tau^{1/2}}\]
for all $f \in \mathcal{S}(\R^d)$ and $\lambda = \tau^2$.
\end{lemma}

\begin{proof}
By the same argument as in the proof of the lemma \ref{lem:Yast-Bourgain_spaces}, it is enough to show that for $k \in I$ we have
\begin{equation}
\|P_k f\|_{L^{p_d}} \lesssim \||q_\tau|^{1/2}\widehat{P_k f}\|_{L^2}.
\label{in:critical}
\end{equation}
Let $g$ be as in \eqref{id:aux_g}, and write $P_k f$ as in \eqref{id:pkf}.
Applying Bernstein's and Plancherel's identities in the variable $x_d$ and after this Minkowski's inequality, we have
\begin{equation}
\begin{aligned}
& 2^{-k(\frac{1}{2} - \frac{1}{p_d})} \|P_k f\|_{L^{p_d}(\R^d)} \\
& \quad \lesssim \Big( \int_{\R^{d - 1}} \big\| \int_0^\infty \int_{\mathbb{S}^{d - 2}} e^{i\rho x^\prime \cdot \theta} \frac{\widehat{P_kg}(\rho\theta, \centerdot)}{|q_\tau (\rho \theta, \centerdot)|^\frac{1}{2}} \, \dd S(\theta) \rho^{d - 2} \, \dd \rho  \big\|_{L^2(\R)}^{p_d} \, \dd x^\prime \Big)^\frac{1}{p_d}\\
& \quad \lesssim \Big\| \int_0^\infty \Big( \int_{\R^{d - 1}} \big| \int_{\mathbb{S}^{d - 2}} e^{i\rho x^\prime \cdot \theta} \frac{\widehat{P_kg}(\rho\theta, \centerdot)}{|q_\tau (\rho \theta, \centerdot)|^{1/2}} \, \dd S(\theta) \big|^{p_d} \, \dd x^\prime \Big)^\frac{1}{p_d}  \rho^{d - 2} \, \dd \rho \Big\|_{L^2(\R)}.
\end{aligned}
\label{in:Zpre_EXTtrace}
\end{equation}
As in the proof of the lemma \ref{lem:Yast-Bourgain_spaces}, we change variables $\rho x^\prime = y^\prime$ so that
\begin{align*}
\Big( \int_{\R^{d - 1}} \big| \int_{\mathbb{S}^{d - 2}} & e^{i\rho x^\prime \cdot \theta} \frac{\widehat{P_kg}(\rho\theta, \centerdot)}{|q_\tau (\rho \theta, \centerdot)|^{1/2}} \, \dd S(\theta) \big|^{p_d} \, \dd x^\prime \Big)^\frac{1}{p_d} \\
& = \rho^{(1 - d)/p_d} \Big( \int_{\R^{d - 1}} \big| \int_{\mathbb{S}^{d - 2}} e^{iy^\prime \cdot \theta} \frac{\widehat{P_kg}(\rho\theta, \centerdot)}{|q_\tau (\rho \theta, \centerdot)|^{1/2}} \, \dd S(\theta) \big|^{p_d} \, \dd y^\prime \Big)^\frac{1}{q_d}.
\end{align*}
Applying the extension version of the Tomas--Stein theorem ---valid here for $d \geq 3$ since $p_d = q_{d - 1}$--- we have that the right-hand side of the previous identity can be bounded so that the inequality \eqref{in:Zpre_EXTtrace} becomes
\begin{equation*}
\|P_k f\|_{L^{p_d}(\R^d)} \lesssim 2^{k(\frac{1}{2} - \frac{1}{p_d})} \Big\| \int_0^\infty \rho^{\frac{1 - d}{p_d}} \Big( \int_{\mathbb{S}^{d - 2}} \frac{|\widehat{P_kg}(\rho\theta, \centerdot)|^2}{|q_\tau (\rho \theta, \centerdot)|} \, \dd S(\theta) \Big)^\frac{1}{2}  \rho^{d - 2} \, \dd \rho \Big\|_{L^2(\R)}.
\end{equation*}
As we noted in the proof of the lemma \ref{lem:Yast-Bourgain_spaces},  
$|q_\tau (\rho \theta, \xi_d)|$ does not depend on $\theta$, and consequently, applying the Cauchy--Schwarz inequality to the integration 
in $\rho$, we have that the norm on the right-hand side of the previous inequality is bounded by a constant multiple of
\[\sup_{|\xi_d| \leq 2^{k + 1}} \Big( \int_0^{2^{k+1}} \frac{\rho^{2/p_d}}{|q_\tau (\rho \theta, \xi_d)|} \, \dd \rho \Big)^\frac{1}{2} \Big\| \int_0^\infty \int_{\mathbb{S}^{d - 2}}  |\widehat{P_kg}(\rho\theta, \centerdot)|^2 \, \dd S(\theta) \rho^{d - 2} \, \dd \rho \Big\|_{L^2(\R)}. \]
Since $k \in I$, one can check that the first term of the previous 
product is bounded by $\tau^{1/p_d-1/2}$. 
Thus, we end up with the inequality
\[\|P_k f\|_{L^{p_d}} \lesssim 2^{k(1/2 - 1/p_d)} \tau^{1/p_d-1/2} \| |q_\tau|^{1/2} \widehat{P_kf} \|_{L^2}.  \]
Since $k \in I$ and $2^k \simeq \tau$, we get the inequality \eqref{in:critical}.
\end{proof}

\begin{corollary}[Haberman] \label{cor:H} \sl Assume $d \geq 3 $. There exists a constant $C>0$ depending on $d$ so that
\[ \| f \|_{L^{p_d}} \leq C \| f \|_{\dot{Y}_\tau^{1/2}}\]
for all $f \in \mathcal{S}(\R^d)$.
\end{corollary}

\begin{proof}
It is a consequence of the lemmas \ref{lem:Lpembedding} and \ref{lem:Zast-Bourgain_spaces}.
\end{proof}

\begin{lemma}\label{lem:Z-Bourgain_spaces}\sl Whenever $d \geq 3 $, there exists a constant $C>0$ depending on $d$ so that
\[\| f \|_{\dot{Y}_\tau^{-1/2}} \leq C \| f \|_{Z_{\lambda,p_d^\prime}}\]
for all $f \in \mathcal{S}(\R^d)$ and $\lambda = \tau^2$.
\end{lemma}
\begin{proof}
It follows from the lemma \ref{lem:Zast-Bourgain_spaces} by a standard duality argument.
\end{proof}

\begin{corollary}[Kenig--Ruiz--Sogge] \sl Assume $d \geq 3 $. There exists a constant $C>0$ depending on $d$ so that
\[\| (\Delta + 2 \tau \partial_{x_d} + \tau^2)^{-1} f \|_{L^{p_d}} \leq C \| f \|_{L^{p_d^\prime}}\]
for all $f \in \mathcal{S}(\R^d)$.
\end{corollary}

\begin{proof}
It follows from the identity \eqref{id:isometry}, the corollary \ref{cor:H} and the lemmas \ref{lem:Z-Bourgain_spaces} and \ref{lem:Lp*embedding}.
\end{proof}

\appendix

\section{The functional analytical framework}\label{app:framework}
Here we prove the propositions \ref{prop:density}, 
\ref{prop:Xlambda} and \ref{prop:Xlambdaast} which describe 
some basic properties of the functional spaces $X_\lambda$ and
$X_\lambda^\ast$. As we see, these propositions will be immediately 
derived from some properties related to the spaces $Y_\lambda$, $Y_\lambda^\ast$, $Z_\lambda$ and $Z_\lambda^\ast$.

\begin{lemma}\label{lem:Yspaces}\sl The pair
$(Y_\lambda, \| \centerdot \|_{Y_\lambda})$
is a Banach space and its dual is 
isomorphic to $(Y_\lambda^\ast, \| \centerdot \|_{Y_\lambda^\ast})$. 
The Schwartz class $\mathcal{S}(\R^d)$ is dense in $Y_\lambda$ and 
$Y_\lambda^\ast$ with their corresponding norms.
\end{lemma}

\begin{lemma}\label{lem:Zspaces}\sl The pair
$(Z_{\lambda,p^\prime}, \| \centerdot \|_{Z_{\lambda,p^\prime}})$
is a reflexive Banach space and its dual is 
isomorphic to $(Z_{\lambda, p}^\ast, \| \centerdot \|_{Z_{\lambda, p}^\ast})$ with $p$ and $p^\prime$ duals. 
The Schwartz class $\mathcal{S}(\R^d)$ is dense in $Z_{\lambda,p^\prime}$ and 
$Z_{\lambda, p}^\ast$ with their corresponding norms.
\end{lemma}

Note that $2^k \simeq \lambda^{1/2}$  when $k \in I$,
$ m_\lambda(\xi)^{1/2} |\widehat{P_{<I} f}(\xi)| \simeq \lambda^{1/2} |\widehat{P_{<I} f}(\xi)|$,
and $ m_\lambda(\xi)^{1/2} |\widehat{P_k f}(\xi)| \simeq 2^{k} |\widehat{P_k f}(\xi)|$
when $k > k_\lambda + 1$. Thus, the norms of the spaces $Y_\lambda$, 
$Y_\lambda^\ast$, $Z_{\lambda,p^\prime}$ and $Z_{\lambda, p}^\ast$ can be re-written 
similarly to the norms of non-homogeneous Besov spaces with different 
weights and norms on the critical scales $k \in I$. This remark is 
the key to justify that these spaces are Banach and
$\mathcal{S}(\R^d)$ is dense with respect to the corresponding 
topologies. The duality also works because of the same principle  
---since the norms in the corresponding pieces are taken to be dual. To be more precise, note that $\vvvert \centerdot \vvvert_\ast$ is the dual norm of $\vvvert \centerdot \vvvert$ and not the other way around, while $\| \centerdot \|_{L^{q_d^\prime}}$ and $\| \centerdot \|_{L^{q_d}}$ are dual of each other. Hence, $Z_{\lambda,p^\prime}$ is reflexive and $Y_\lambda$ is not.

Now, we show how to derive the propositions 
\ref{prop:density}, \ref{prop:Xlambda} and \ref{prop:Xlambdaast} 
stated in the section \ref{sec:scattering}. Start by the first of 
these three propositions. The density of
$\mathcal{S}(\R^d)$ in $Y_\lambda$ and $Z_\lambda$ is explicitly 
stated in the lemmas \ref{lem:Yspaces} and \ref{lem:Zspaces}, in 
particular, the density also holds for
$X_\lambda = Y_\lambda + Z_\lambda$ with its corresponding norm. This 
proves the proposition \ref{prop:density}. Now, we turn to the 
proposition \ref{prop:Xlambda}.
Since $(Y_\lambda, \| \centerdot \|_{Y_\lambda})$ and
$(Z_\lambda, \| \centerdot \|_{Z_\lambda})$ are 
Banach spaces and $Y_\lambda$ and $Z_\lambda$ are subspaces of
$\mathcal{S}^\prime (\R^d)$, we have by Theorem 1.3 in 
\cite{zbMATH00192914} that
$(X_\lambda, \| \centerdot \|_{X_\lambda} )$ is a Banach space. The 
identity \eqref{id:characterizationXlambda} 
is a standard property of Banach spaces (see Corollary 1.4 in 
\cite{zbMATH05633610}). This concludes the proof of the proposition 
\ref{prop:Xlambda}. Finally, let us focus on the last of these three 
propositions. It is a well-known fact ---since $\mathcal{S}(\R^d)$ is 
dense in $Y_\lambda$ and $Z_\lambda$--- that $(X_\lambda^\ast, \| \centerdot \|_{X_\lambda^\ast} )$ is isomorphic to the space $Y_\lambda^\ast \cap Z_\lambda^\ast$ endowed with the norm
$ \max \{ \| \centerdot \|_{Y_\lambda^\ast} , \| \centerdot \|_{Z_\lambda^\ast} \}$ (see 2 in the section Exercises and Further Results 
for Chapter 3 of the book \cite{zbMATH00192914}). 
Note that this later space is actually isomorphic 
to the space described in the proposition \ref{prop:Xlambdaast} 
endowed with the norm \eqref{term:Xlambdaast}. To finish the proof of 
this proposition, it is enough to check the density of
$\mathcal{S}(\R^d)$ in $X_\lambda^\ast$ with its corresponding norm. 
Note that this holds because $\mathcal{S}(\R^d)$ is dense in
$Y_\lambda^\ast$ and $Z_\lambda^\ast$, and the norm $\| \centerdot \|_{X_\lambda^\ast}$ is equivalent to $ \max \{ \| \centerdot \|_{Y_\lambda^\ast} , \| \centerdot \|_{Z_\lambda^\ast} \}$.

%

\bibliography{references}{}
\bibliographystyle{plain}

\end{document}